\documentclass[11pt, reqno]{amsart}

\usepackage[sort,nospace,compress,noadjust]{cite}
\usepackage{newtxtext}
\usepackage{newtxmath}

\usepackage{colortbl}
\usepackage{booktabs}

\usepackage{tikz}
\usepackage{tkz-graph}
\usetikzlibrary{arrows.meta}
\usepackage{tikz-cd}

\usepackage{xspace}
\usepackage[hidelinks]{hyperref}
\usepackage{mathtools}

\usepackage{enumitem}

\usepackage[foot]{amsaddr}

\newcommand{\bQ}{\mathbb{Q}}

\newcommand{\bA}{\mathbf{A}}
\newcommand{\bB}{\mathbf{B}}
\newcommand{\bC}{\mathbf{C}}
\newcommand{\bD}{\mathbf{D}}
\newcommand{\bL}{\mathbf{L}}
\newcommand{\bK}{\mathbf{K}}
\newcommand{\bX}{\mathbf{X}}
\newcommand{\bP}{\mathbf{P}}

\newcommand{\cC}{\mathcal{C}}

\newcommand{\sA}{\mathscr{A}}

\newcommand{\fA}{\mathfrak{A}}
\newcommand{\fB}{\mathfrak{B}}
\newcommand{\fT}{\mathfrak{T}}

\newcommand{\tn}{\overline{n}}

\newcommand{\Fraisse}{Fra\"\i{}ss\'e\xspace}
\newcommand{\Nesetril}{Ne\v{s}et\v{r}il\xspace}
\newcommand{\Masulovic}{Ma\v{s}ulovi\'c\xspace}

\newcommand{\restr}{\mathord{\upharpoonright}}
\newcommand{\dotcup}{\mathop{\dot\cup}}

\newcommand{\End}{\operatorname{End}}
\newcommand{\Age}{\operatorname{Age}}
\newcommand{\VGamma}{\operatorname{V}(\Gamma)}
\newcommand{\V}{\operatorname{V}}
\newcommand{\EGamma}{\operatorname{E}(\Gamma)}
\newcommand{\E}{\operatorname{E}}
\newcommand{\VDelta}{\operatorname{V}(\Delta)}
\newcommand{\im}{\operatorname{Im}}
\DeclareMathOperator{\dom}{dom}

\newcommand{\SetColorBlue}%
{\tikzset{EdgeStyle/.append style = {color = blue}}%
\tikzset{VertexStyle/.append style = {color = blue}}
}
\newcommand{\SetVertexCircle}{\tikzset{VertexStyle/.append style = {shape = circle, minimum size = 2pt, fill = none}}}
\newcommand{\SetVertexNone}{\tikzset{VertexStyle/.append style = {shape = coordinate, minimum size = 2pt, fill = none}}\presetkeys [GR] {vertex} {NoLabel = true}{}}
\newcommand{\SetupArcStyle}%
{\tikzset{EdgeStyle/.style = {->}}\tikzset{>=stealth'}}
\def\scf{1}

\theoremstyle{plain}
\newtheorem{theorem}{Theorem}[section]
\newtheorem{proposition}[theorem]{Proposition}
\newtheorem{lemma}[theorem]{Lemma}
\newtheorem{corollary}[theorem]{Corollary}
\theoremstyle{definition}
\newtheorem{definition}[theorem]{Definition}
\newtheorem{example}[theorem]{Example}
\newtheorem*{remark}{Remark}
\newtheorem{observation}[theorem]{Observation}

\title[Homomorphism homogeneous oriented graphs]{The classification of homomorphism homogeneous oriented graphs}

\author[B.\,Pavlica]{Bojana Pavlica$^1$}
\address{$^1$Department of Mathematics and Informatics\\
        Faculty of Sciences\\
        University of Novi Sad\\
        Trg Dositeja Obradovi\'ca 3\\
        21000 Novi Sad\\
        Serbia} 
\email{bojana@dmi.uns.ac.rs}

\author[Ch.\,Pech]{Christian Pech$^2$}
\address{$^2$Institute of Mathematics\\
        Czech Academy of Sciences\\
        \v{Z}itn\'a 25\\ 
        115\,67 Praha 1\\ 
        Czech Republic}
\email{pech@math.cas.cz}
\thanks{The research of the second author was supported by GA~\v{C}R (Czech Science Foundation) grant EXPRO 20-31529X}

\author[M.\,Pech]{Maja Pech$^{1\ast}$} 
\email{maja@dmi.uns.ac.rs}
\thanks{The third author gratefully acknowledges the financial support of the Ministry of Science, Technological Development and Innovation of the Republic of Serbia (Grants No. 451-03-66/2024-03/ 200125 \& 451-03-65/2024-03/200125)}

\thanks{$^\ast$ corresponding author}

\begin{document}
\begin{abstract}
    The modern theory of homogeneous structures begins with the work of Roland \Fraisse.  The theory developed in the last seventy years is placed in the border area between combinatorics, model theory, algebra, and analysis. We turn our attention to its combinatorial pillar, namely, the work on the classification of structures for given homogeneity types, and focus onto the homomorphism homogeneous ones, introduced in 2006 by Cameron and \Nesetril. An oriented graph is called homomorphism  homogeneous if every homomorphism between finite induced subgraphs extends to an endomorphism.
    In this paper we present a complete classification of the countable homomorphism homogeneous oriented graphs. Among these we identify those that are polymorphism homogeneous. Here an oriented graph is called \emph{polymorphism homogeneous} if each of its finite powers is homomorphism homogeneous.
\end{abstract}
\maketitle

\section{Introduction}

The classical notion of homogeneity was introduced by Roland \Fraisse in the early fifties and has been thoroughly studied during the last seventy years (see \cite{Mac11}). Recall that a relational structure is \emph{homogeneous} if every isomorphism between finite  substructures extends to an automorphism. 
At the beginning of this century, the notion of homomorphism homogeneity was introduced by Cameron and \Nesetril in their seminal paper on this topic \cite{CamNes06}. A relational structure is called \emph{homomorphism homogeneous} (shortly HH) if every homomorphism between two of its  finite substructures extends to an endomorphism of the structure in question. In the mentioned paper, this phenomenon was studied for simple graphs and posets, with a number of inspiring and challenging questions posed. This initiated the research on the classification of countable HH relational structures with exactly one binary relation. First results were obtained relatively quickly by Cameron and Lockett \cite{CamLoc10}, and \Masulovic \cite{Mas07} who completely classified HH posets with strict and reflexive order relations, respectively. On the other hand, the classification of countable HH simple graphs is still wide open, and the subject of ongoing research \cite{AraHar19b,AraHar20}. Interestingly, already when allowing loops in undirected graphs the problem of classifying even finite HH structures becomes in a sense untractable.
Namely, it was shown by Rusinov and Schweizer in \cite{RusSch10} that the problem to decide whether a finite graph with loops allowed is HH is coNP-complete. Luckily it seems that HH digraphs with antisymmetric arc relation are more amenable to classification efforts: finite HH tournaments with loops allowed were classified in \cite{IliMasRaj08} (the countable case was completed in \cite{FelPecPec20}). Finite HH uniform oriented graphs (i.e. oriented graphs  with all / with no loops) were classified in \cite{Mas15}. Further insights concerning HH oriented graphs were obtained by Coleman (see \cite{Col20}).

In this paper we set out to classify the countable HH oriented graphs without loops (from now on, instead of ``oriented graphs without loops'', we will write ``oriented graphs''). At this point it is worth mentioning that the analogous problem, namely the classification of homogeneous oriented graphs was carried out by Cherlin in \cite{Cher1998}. It fills a whole volume of the Memoirs of the AMS, and already this fact was quite intimidating at the outset of the project. Fortunately, classifications of HH structures and of homogeneous structures are rather different in nature (but not completely unrelated). To our great delight it turned out that  the classification of HH oriented graphs may be stratified by the countable homogeneous tournaments classified by Lachlan and Woodrow \cite{Lac84,Woo77}. The key insight is that every countable HH oriented graph has a unique (up to isomorphism) homogeneous, HH core (as a consequence of a more general result from \cite{PecPec16b}), and that HH oriented graphs that are cores must be tournaments (and thus homogeneous). This allowed us to finish the classification of countable HH oriented graphs (Theorem~\ref{MainHH} and Theorem~\ref{wccsametour}). 

Closely related to the notion of homomorphism homogeneity is the notion of polymorphism homogeneity. We say that a relational structure is \emph{polymorphism homogeneous} (shortly PH) if each of its finite  powers is homomorphism homogeneous (see Definition~\ref{def:PH}). The notion of polymorphism homogeneity has its origins in universal algebra where it has been used when studying polymorphism clones of countable homogeneous structures (\cite{PecPec18}), and phenomena in universal algebraic geometry (\cite{TotWal21}). First steps towards the classification theory of PH structures were done in \cite{FarJS15,FelPecPec20,PecPec15}.
Concerning the scope of this paper, we can mention that countable PH simple graphs, PH posets with strict and reflexive order relation, and countable PH tournaments (with loops allowed) are completely classified (see \cite{PecPec15,FelPecPec20}). This paper adds the countable oriented graphs to this list (Theorem~\ref{mainphcon} and Theorem~\ref{mainphdiscon}).

Before starting with the exposition of our results, let us fix some notions and notation.
Generally, we are using the usual graph theoretic terminology (see \cite{ChaLesZha16}).
Recall that an \emph{oriented graph} is a digraph such that in between any two distinct vertices there is at most one arc. Formally, an oriented graph may be modelled as a pair $\Gamma=(V,E)$, where
\begin{itemize}
    \item $V$ is a non-empty set of vertices,
	\item $E\subseteq V^2$ is an asymmetric relation, i.e. $\forall x,y\in V:\, (x,y)\in E\Rightarrow (y,x)\notin E$.
\end{itemize}
The elements of $E$ are called \emph{arcs} of $\Gamma$. For every oriented graph $\Gamma$, denote the vertex set and arc set of $\Gamma$ by $\VGamma$ and $\EGamma$, respectively.

For two subsets $B_1$ and $B_2$ of $\VGamma$ we write $B_1\rightarrow B_2$ if $(x,y)\in \EGamma$ for all  $x\in B_1$ and $y\in B_2$. Instead of $\{b_1\}\rightarrow B_2$, $B_1\rightarrow \{b_2\}$, and $\{b_1\}\rightarrow\{b_2\}$ we write $b_1\rightarrow B_2$, $B_1\rightarrow b_2$, and  $b_1\rightarrow b_2$, respectively.

For two vertices $x,y\in \VGamma$ we write $x \sim y$ if either $x\rightarrow y$ or $y\rightarrow x$. 
Correspondingly we write $x\not\sim y$ if it does not hold that $x\sim y$.

The notation $x\sim^\ast y$ means that there is a \emph{semi-walk} from $x$ to $y$ in $\Gamma$. In other words, there is a finite sequence $z_0,\dots, z_k$ (where $k\ge 0$) such that $x=z_0$, $y=z_k$, and for all $0\le i< k$ we have $z_i\sim z_{i+1}$.  
Clearly, $\sim^\ast$ is an equivalence relation on the vertex set of $\Gamma$. Its equivalence classes are the \emph{weakly connected components} of $\Gamma$. If $\Gamma$ has only one weakly connected component, then we say that it is \emph{weakly connected}.  Otherwise, we say that it is \emph{disconnected}.

Throughout the paper, we are going to classify oriented graphs by whether or not they contain certain configurations. Here a \emph{configuration} is nothing else but a finite (unlabelled) oriented graph, and a given oriented graph $\Gamma$ contains a configuration $\Delta$ if some induced subgraph of $\Gamma$ is isomorphic to $\Delta$. By $\Age(\Gamma)$ we denote the class of all configurations contained in $\Gamma$.

We say that two oriented graphs $\Gamma_1$ and $\Gamma_2$ are \emph{homomorphism equivalent} if there is an oriented graph homomorphism from $\Gamma_1$ to $\Gamma_2$ and vice versa.

Finally, throughout the paper $\omega$ denotes the set of non-negative integers. Moreover we use the convention that countable means finite or countably infinite.

Our results are presented in the four sections that follow. In Section~\ref{sec2} and \ref{sec3} we classify weakly connected homomorphism homogeneous and polymorphism homogeneous oriented graphs, respectively. Sections \ref{sec4} and \ref{sec5} handle the cases of disconnected oriented graphs.  The paper ends with an informal discussion that compares the classification of HH oriented graphs with Cherlin's classification of homogeneous oriented graphs (see Section~\ref{discussions}), and outlines perspectives for future research (see Section~\ref{future}).  

\section{Weakly connected homomorphism homogeneous oriented graphs}\label{sec2}

Our approach to the classification of weakly connected homomorphism homogeneous oriented graphs is based on the use of cores:

\begin{definition}
An oriented graph $\Gamma$ is called a \emph{core} if every endomorphism of $\Gamma$ is a self-embedding.
An oriented graph $\Delta$ is a \emph{core of the oriented graph} $\Gamma$ if it satisfies the following conditions:
\begin{enumerate}
	\item $\Delta$ is a core,
	\item $\Delta$ is an induced subgraph of $\Gamma$, and
	\item there exists $h\in \End(\Gamma)$ such that $h[\VGamma]\subseteq \VDelta$.
\end{enumerate} 		
\end{definition}
The following observation is a special case of a model theoretic result about cores of relational structures: 
\begin{proposition}[{\cite[Corollary 6.7]{PecPec16b}}]
Every countable homomorphism homogeneous oriented graph has, up to isomorphism, a unique homomorphism homogeneous core.
Moreover, this core is homogeneous.
\end{proposition}
\begin{proof}
    First note that oriented graphs are special cases of relational structures with one binary relation. 

    \cite[Corollary 6.7]{PecPec16b} says that the claim of the proposition holds for all countable weakly oligomorphic, homomorphism homogeneous relational structures. 

    It remains to observe that countable homomorphism homogeneous oriented graphs are weakly oligomorphic. However, this follows from \cite[Proposition 2.3]{MasPec11} and the fact that oriented graphs, considered as relational structures, have a finite signature.  
\end{proof}

The strategy of classifying weakly connected homomorphism homogeneous oriented graphs  is now clear, and can be divided in two steps: 
\begin{description}
	\item[Step 1] Identify all weakly connected homogeneous homomorphism homogeneous oriented graphs that are cores.
	\item[Step 2] For each of them classify weakly connected homomorphism homogeneous oriented graphs with a given core. 
\end{description}

Concerning Step 1, there is a straightforward way of identifying cores among oriented graphs that are both homogeneous and homomorphism homogeneous, based on the classification of homogeneous tournaments by Lachlan  and Woodrow in \cite{Lac84,Woo77}:
\begin{lemma}
    Homomorphism homogeneous oriented graphs that are cores must be tournaments.
\end{lemma}
\begin{proof}
    Let $\Gamma$ be a homomorphism homogeneous oriented graph that is a core. Suppose that $\Gamma$ is not a tournament. Then there exist $u,v\in \VGamma$ that are not connected by an arc. However, then the mapping $f\colon\{u,v\}\to \VGamma$ that maps both $u$ and $v$  to $v$ is a local homomorphism of $\Gamma$ (in general a \emph{local homomorphism} of $\Gamma$ is a homomorphism from a finite subgraph of $\Gamma$ to $\Gamma$). As $\Gamma$ is homomorphism homogeneous, $f$ extends to an endomorphism of $\Gamma$ that is not a self-embedding --- a contradiction. 
\end{proof}

It is easy to see that every tournament is a core. Moreover,  a countable tournament is homomorphism homogeneous if and only if it is homogeneous. 
According to \cite{Lac84} all countable homogeneous tournaments are isomorphic to one of the tournaments from the following list:
\begin{itemize}
    \item[$I_1$:] the tournament that has just one vertex and no arc,
    \item[$C_3$:] the oriented cycle of length $3$,
    \item[$(\bQ,<)$:] the rational numbers with the strict order,
    \item[$S(2)$:]\label{s2def} the countable circular tournament.  It is obtained by choosing a countable dense subset $S$ of the unit circle in such a way that no two points of $S$ are antipodal. For any two points $x,y\in S$ an arc is  drawn from $x$ to $y$ whenever the angle traversed starting from $x$ and going counter-clock wise to $y$ is less than $\pi$,
    \item[$T^\infty$:] the countable universal homogeneous tournament --- the \Fraisse limit of the class of all finite tournaments.
\end{itemize}
\begin{remark}
    In the description of the classification of homogeneous tournaments above the term ``\Fraisse limit'' appears.  This is because it was shown by \Fraisse (see \cite{Fra53}) that countable homogeneous structures are uniquely determined (up to isomorphism) by their age. Therefore the age of a countable homogeneous structure is usually called a \emph{\Fraisse class} and the unique countable homogeneous structure whose age is a given \Fraisse class $\mathcal{C}$ is called the \emph{\Fraisse limit} of $\mathcal{C}$ (for a modern statement of this result, see \cite{Hod97}). It should be mentioned that the situation for homomorphism homogeneous structures is quite different. It is possible that two non-isomorphic countable homomorphism homogeneous structures have the same age. For example, every countably infinite  chain $(C,\le)$ is homomorphism homogeneous (see \cite[Theorem 4.5]{Mas07}) while its age consists of all finite chains. Thus, there are continuum many pairwise non-isomorphic countably infinite  chains all sharing the same age. However,  any two homomorphism homogeneous structures with the same age are homomorphism-equivalent (see \cite[Lemma 3.6]{PecPec16b}).
\end{remark}
\begin{corollary}\label{poss_cores}
	The only possible cores of homomorphism homogeneous oriented graphs are
	\begin{equation*} 
 I_1, C_3, (\bQ, <), S(2), \text{\,and } T^{\infty}.
 \end{equation*}
\end{corollary}

Now we are ready to proceed with Step 2 of our strategy. Once the possible cores are identified, we move to the classification of weakly connected homomorphism homogeneous oriented graphs with a given core. It is an easy observation that $I_1$ and $(\bQ, <)$ are the cores of acyclic weakly connected homomorphism homogeneous oriented graphs, and this defines our classification strategy --- we conduct our further considerations in two directions, analyzing separately acyclic oriented graphs, and those that contain at least one cycle.

\subsection*{Acyclic homomorphism homogeneous oriented graphs} 
The crucial observation that enables the classification in this case is the following result:
\begin{proposition}\label{AcyclicTransitive}
	If $\Gamma$ is an acyclic homomorphism homogeneous oriented graph, then its arc relation is transitive.
\end{proposition}
The proof of this proposition is based on a general property of homomorphism homogeneous oriented graphs that was firstly realized by Coleman:
\begin{lemma}[{\cite[Corollary 5.12]{Col20}}]\label{NoInducedChain}
Let $\Gamma$ be a homomorphism homogeneous oriented graph. Then $\Gamma$ does not contain the configuration
\[
\begin{tikzpicture}[scale=\scf,baseline=0pt]
    \SetVertexCircle		
    \SetupArcStyle
	\SetVertexNoLabel
	\Vertex[x=1,y=0]{B}
	\Vertex[x=2,y=0]{C}
	\Vertex[x=0,y=0]{A}
    \Edges(A,B,C)
\end{tikzpicture}\,.
\]	
\end{lemma}
\begin{proof}
Suppose that $\Gamma$ is a homomorphism homogeneous oriented graph that contains 
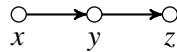
\[
\begin{tikzpicture}[scale=\scf,baseline=0pt]
    \GraphInit[vstyle=Classic]
    \SetUpVertex[Lpos=-90]
    \SetVertexCircle		
    \SetVertexLabel
    \SetVertexMath
    \SetupArcStyle
   	\Vertex[x=1,y=0,L={y}]{y}
   	\Vertex[x=2,y=0,L={z}]{z}
   	\Vertex[x=0,y=0,L={x}]{x}
    \Edges(x,y,z)
\end{tikzpicture}
\]
    as an induced subgraph. Consider the local homomorphism $f\colon \{x,z\}\to \VGamma$ of $\Gamma$ given by $f:=\begin{psmallmatrix}  x & z \\ x & x\end{psmallmatrix}$. Since $\Gamma$ is homomorphism homogeneous, $f$ can be extended to a local homomorphism $\hat{f}\colon \{x,y,z\}\to \VGamma$. Observe that from $x\to y\to z$ it follows $\hat{f}(x)\to \hat{f}(y)\to \hat{f}(z)$, i.e.\ $x\to \hat{f}(y)\to x$, so either $\hat{f}(y)=x$ or both $x\to \hat{f}(x)$ and $\hat{f}(x)\to x$. In both situations we arrive at a contradiction with the asymmetry of $\EGamma$.
    Hence, $\Gamma$ has no induced subgraphs of the given shape.		
\end{proof}

\begin{proof}[Proof of Proposition~\ref{AcyclicTransitive}]
Suppose that $\Gamma$ is a homomorphism homogeneous acyclic oriented graph, and let $x,y,z\in \VGamma$ be such that $x\to y\to z$. By Lemma~\ref{NoInducedChain}, there is an oriented edge between $x$ and $z$. Since $\Gamma$ is acyclic, it follows that $x\to z$. Hence, $\EGamma$ is transitive. 
\end{proof}
	
It is clear that transitive oriented graphs can be viewed as strict posets, so we obtain:
\begin{corollary}
	Every acyclic homomorphism homogeneous oriented graph is a strict poset.
\end{corollary} 
The homomorphism homogeneous strict posets were completely classified by Cameron and Lockett \cite{CamLoc10}, and this enables us to give the classification in this case:

\begin{proposition}[{\cite[Proposition 15]{CamLoc10}}]\label{acyclicHH} Weakly connected acyclic homomorphism homogeneous oriented graphs are
\begin{enumerate}
	\item $I_1$, 
	\item $(\bQ, <)$,
	\item trees with no minimal elements such that no finite subset of vertices has a maximal lower bound,
	\item dual trees with no maximal elements such that no finite subset of vertices has a minimal upper bound,	
	\item posets such that: \begin{itemize}
 \item every finite subset of vertices is bounded from above and from below
 \item no finite subset of vertices has a maximal lower bound or a minimal upper bound	
 \item no $X_4$-set has a midpoint,
 \end{itemize}
 \item all extensions of the countable universal homogeneous strict  poset.	
\end{enumerate}
\end{proposition}	
\begin{remark}
A poset $(P,<)$ is called a \emph{tree} if  for every $x\in P$ we have that the set of elements in $P$ below $x$ forms a chain. Moreover, $(P,<)$ is called a \emph{dual tree} if $(P,>)$ is a tree. Finally, an \emph{$X_4$-set} in $(P,<)$ is a $4$-element subset of $P$ that induces a subposet of the shape
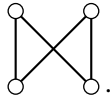
\[
   	\begin{tikzpicture}[scale=\scf,baseline=-2pt]
        \SetVertexCircle		
		\SetVertexNoLabel
   		\Vertex[x=0,y=0]{C1}
   		\Vertex[x=1,y=0]{C2}
   		\Vertex[x=0,y=1]{C3}
   		\Vertex[x=1,y=1]{C4}
        \Edge(C1)(C3)
        \Edge(C2)(C3)
        \Edge(C1)(C4)
        \Edge(C2)(C4)
    \end{tikzpicture}\,.
\]
We say that an $X_4$-set has a \emph{midpoint} if there is a fifth element that together with the $X_4$-set induces a subposet of the shape: 
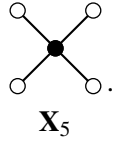
\[\label{X5}
   	\begin{tikzpicture}[scale=\scf,baseline=-2pt]
        \SetVertexCircle		
		\SetVertexNoLabel
   		\Vertex[x=0,y=0]{B1}
   		\Vertex[x=1,y=0]{B2}
   		\Vertex[x=0.5,y=0.5]{A}
   		\Vertex[x=0,y=1]{B3}
   		\Vertex[x=1,y=1]{B4}
		\AddVertexColor{black}{A}
        \Edge(B1)(A)
        \Edge(B2)(A)
        \Edge(A)(B3)
        \Edge(A)(B4)
        \node at (0.5,-0.5) {$\bX_5$};
    \end{tikzpicture}\,.
\]
Recall also that the countable universal homogeneous strict poset is the \Fraisse limit of the class of all finite posets. Finally, a poset $(A,<_2)$ is called an \emph{extension} of a poset $(A,<_1)$ if $<_1$ is a subset of $<_2$.  
\end{remark}

\subsection*{Homomorphism homogeneous oriented graphs that contain cycles.} Again, there is an easy but important observation that directs our strategy for the classification.
	
\begin{lemma}\label{alwaysC3}
Every induced oriented cycle in a homomorphism homogeneous oriented graph is isomorphic to $C_3$.
\end{lemma}

\begin{proof}
	Let $\Gamma$ be a homomorphism homogeneous oriented graph. From Lemma~\ref{NoInducedChain} it follows that $\Gamma$ contains no induced oriented path of length 2, so $\Gamma$ cannot contain an induced oriented cycle of length greater than 3. Hence, every induced oriented cycle in $\Gamma$ has to be isomorphic to $C_3$.
\end{proof}

Since homomorphism homogeneous tournaments are just the homogeneous tournaments that were listed in front of Corollary~\ref{poss_cores}, we turn our attention to homomorphism homogeneous oriented graphs that are not tournaments, but contain cycles. The possible cores of such graphs are $C_3$, $S(2)$, or $T^\infty$. As a first step we would like to understand the non-arc relation $\not\sim$ in these cases. To this end we analyse   subconfigurations on three and four vertices that may or may not occur in such graphs, with the goal of showing that $\not\sim$ is an  equivalence relation.


\begin{lemma}[3-vertex configurations]\label{3conf}
	Let $\Gamma$ be a weakly connected homomorphism homogeneous oriented graph that is not a tournament. Then $\Age(\Gamma)$ contains at least one of the following configurations:
	
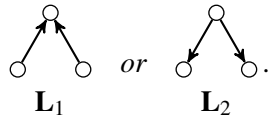
\[
	\begin{tikzpicture}[scale=\scf,baseline=-8pt]
        \GraphInit[vstyle=Classic]
        \SetVertexNoLabel
        \SetVertexCircle
        \SetupArcStyle
   		\Vertex[a=0*120-30,d=0.5]{v}
   		\Vertex[a=1*120-30,d=0.5]{w}
   		\Vertex[a=2*120-30,d=0.5]{u}
        \node at (0,-0.7) {$\bL_1$};
        \Edge(u)(w)
        \Edge(v)(w)
  \end{tikzpicture} \quad\text{or}\quad
  \begin{tikzpicture}[scale=\scf,baseline=-8pt]
        \GraphInit[vstyle=Classic]
        \SetVertexNoLabel
        \SetVertexCircle
        \SetupArcStyle
   		\Vertex[a=0*120-30,d=0.5]{v}
   		\Vertex[a=1*120-30,d=0.5]{w}
   		\Vertex[a=2*120-30,d=0.5]{u}
        \node at (0,-0.7) {$\bL_2$};
        \Edge(w)(u)
        \Edge(w)(v)
  \end{tikzpicture}\,.
\]	  
\end{lemma}
\begin{proof}
	Take $x,y\in \VGamma$ with $x\not\sim y$. Since $\Gamma$ is weakly connected, there exists a non-oriented path between $x$ and $y$. Let $x\sim z_1\sim\cdots\sim z_k\sim y$ be such a path of the shortest length. If $k=1$, then we get $x\sim z_1\sim y$. On the other hand, if $k\geq 2$, then we get $x\sim z_1\sim z_2$, and $x\not\sim z_2$, since the observed path is the shortest one. In both cases, we find $u,v,w\in \VGamma$ such that $u\sim v\sim w$, and $u\not\sim w$, so the possible induced subgraphs are $u\to v\gets w$, $u\gets v\to w$, $ u\gets v \gets w$, and $u\to v\to w$. The last two can be disqualified by Lemma~\ref{NoInducedChain}.  
\end{proof}
\begin{remark}
	Lemma~\ref{3conf} implies that the graph obtained from a homomorphism homogeneous oriented graph by forgetting the orientation of its arcs has diameter at most $2$. This mirrors  \cite[Proposition 1.1(c)]{CamNes06} where it was shown that homomorphism homogeneous graphs have diameter at most $2$. 
\end{remark}

\begin{lemma}[4-vertex configurations]\label{4conf}
	Let $\Gamma$ be a homomorphism homogeneous oriented graph  that contains $C_3$, and that is not a tournament. Then $\Gamma$ contains the following configuration: 
	
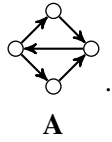
\[
    \begin{tikzpicture}[scale=\scf,baseline=-16pt]
        \GraphInit[vstyle=Classic]
        \SetVertexNoLabel
        \SetVertexCircle
        \SetupArcStyle
   		\Vertex[x=0,y=-0.5]{y}
   		\Vertex[x=0.5,y=0]{z}
   		\Vertex[x=-0.5,y=0]{u}
   		\Vertex[x=0,y=0.5]{x}
        \Edges(z,u,y,z)
        \Edges(u,x,z)
        \node at (0,-1) {$\bA$};
    \end{tikzpicture}\,.
    \]
\end{lemma}
\begin{proof}
	From Lemma~\ref{3conf} we have that $\Gamma$ contains $\bL_1$ or $\bL_2$.
Suppose that it contains $\bL_1$  and consider the following two induced subgraphs of $\Gamma$:
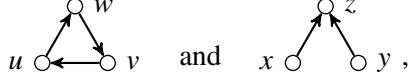
 \[
    \begin{tikzpicture}[scale=\scf,baseline=-8pt]
        \GraphInit[vstyle=Classic]
        \SetVertexCircle
        \SetupArcStyle
        \SetVertexLabel
 		\SetVertexMath
  		\Vertex[a=0*120-30,d=0.5,L={v}]{v}
   		\Vertex[a=1*120-30,d=0.5,L={w}]{w}
   		\Vertex[a=2*120-30,d=0.5,L={u},Lpos=180]{u}
        \Edges(w,v,u,w)
    \end{tikzpicture} \quad\text{and}\quad	
    \begin{tikzpicture}[scale=\scf,baseline=-8pt]
        \GraphInit[vstyle=Classic]
        \SetVertexCircle
        \SetupArcStyle
        \SetVertexLabel
		\SetVertexMath
   		\Vertex[a=0*120-30,d=0.5,L={y}]{v}
   		\Vertex[a=1*120-30,d=0.5,L={z}]{w}
   		\Vertex[a=2*120-30,d=0.5,L={x},Lpos=180]{u}
        \Edge(u)(w)
        \Edge(v)(w)
    \end{tikzpicture},
\]
as well as the local homomorphism $f\colon \{u,v\}\to \VGamma$ of $\Gamma$ given by $f:=\begin{psmallmatrix}  u & v \\ z & y\end{psmallmatrix}$. Since $\Gamma$ is homomorphism homogeneous, $f$ can be extended to a local homomorphism  $\hat{f}\colon \{u,v,w\}\to \VGamma$, with $z=\hat{f}(u)\to \hat{f} (w)\to \hat{f}(v)=y$. Note that $\hat{f}(w)\notin\{x,y,z\}$. This implies that depending on the relation between vertices $x$ and $\hat{f}(w)$ the subgraph induced by $\{x,y,z,\hat{f}(w)\}$ is either of the following:  
 
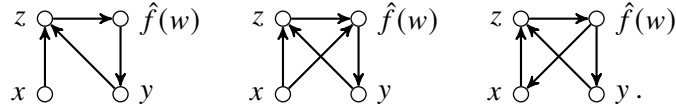
\[
	\begin{tikzpicture}[scale=\scf,baseline=0pt]
        \GraphInit[vstyle=Classic]
        \SetVertexCircle
        \SetupArcStyle
        \SetVertexLabel
		\SetVertexMath
   		\Vertex[x=0,y=-0.5,Lpos=-180,L={x}]{b}
   		\Vertex[x=1,y=0.5,L={\hat{f}(w)},Lpos=0]{f}
   		\Vertex[x=1,y=-0.5,Lpos=0,L={y}]{c}
   		\Vertex[x=0,y=0.5,Lpos=-180,L={z}]{a}
        \Edge(b)(a)
        \Edge(a)(f)
        \Edge(f)(c)
        \Edge(c)(a)
    \end{tikzpicture}\quad 
	\begin{tikzpicture}[scale=\scf,baseline=0pt]
        \GraphInit[vstyle=Classic]
        \SetVertexCircle
        \SetupArcStyle
        \SetVertexLabel
		\SetVertexMath
   		\Vertex[x=0,y=-0.5,Lpos=-180,L={x}]{b}
   		\Vertex[x=1,y=0.5,L={\hat{f}(w)},Lpos=0]{f}
   		\Vertex[x=1,y=-0.5,Lpos=0,L={y}]{c}
   		\Vertex[x=0,y=0.5,Lpos=-180,L={z}]{a}
        \Edge(b)(a)
        \Edge(a)(f)
        \Edge(f)(c)
        \Edge(c)(a)
        \Edge(b)(f)
    \end{tikzpicture}\quad
	\begin{tikzpicture}[scale=\scf,baseline=0pt]
        \GraphInit[vstyle=Classic]
        \SetVertexCircle
        \SetupArcStyle
        \SetVertexLabel
		\SetVertexMath
   		\Vertex[x=0,y=-0.5,Lpos=-180,L={x}]{b}
   		\Vertex[x=1,y=0.5,L={\hat{f}(w)},Lpos=0]{f}
   		\Vertex[x=1,y=-0.5,Lpos=0,L={y\,.}]{c}
   		\Vertex[x=0,y=0.5,Lpos=-180,L={z}]{a}
        \Edge(b)(a)
        \Edge(a)(f)
        \Edge(f)(c)
        \Edge(c)(a)
        \Edge(f)(b)
    \end{tikzpicture}
\]         
Note that the first two cases may not occur, since both graphs contain an induced path of length 2 and are thus excuded by Lemma~\ref{NoInducedChain}. The third graph is isomorphic to $\bA$. 

The case that $\Gamma$ contains $\bL_2$ is handled analogously.
\end{proof}

Note that $\bA$ embeds both $\bL_1$ and $\bL_2$, thus proving:
\begin{corollary}\label{confD}
	Let $\Gamma$ be a homomorphism homogeneous oriented graph that contains $C_3$, and that is not a tournament. Then $\Gamma$ contains configurations  $\bL_1$ and $\bL_2$.
\end{corollary}

\begin{proposition}\label{nonedgeeqrel}
Let $\Gamma$ be a homomorphism homogeneous oriented graph  that contains $C_3$. Then the non-edge relation $\not\sim$ is an equivalence relation.
\end{proposition}
The proof of this claim makes use of the following simple observation:
\begin{observation}\label{S2K}
    $S(2)$ contains configuration $\bK$ given below:
    
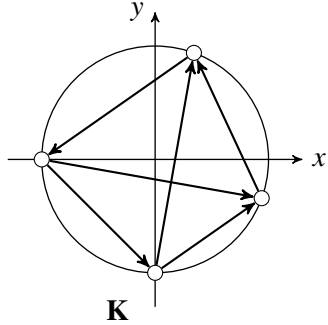
\[
\begin{tikzpicture}[scale=\scf]
    \def\R{1.5}
    \tikzset{EdgeStyle/.style = {->}}
    \tikzset{>=stealth'}		
    \draw[->,line width=0.5] (-1.3*\R,0) -- (1.3*\R,0) node[right] {$x$};
    \draw[->,line width=0.5] (0,-1.3*\R) -- (0,1.3*\R) node[left] {$y$};
    \draw[line width = 0.5] (0,0) circle (\R);
    \node[shape=circle,draw,fill = white,minimum size = 2pt, inner sep = 2pt] (u) at (70:\R){};
    \node[shape=circle,draw,fill = white,minimum size = 2pt, inner sep = 2pt] (x) at (180:\R){};
    \node[shape=circle,draw,fill = white,minimum size = 2pt, inner sep = 2pt] (v) at (270:\R){};
    \node[shape=circle,draw,fill = white,minimum size = 2pt, inner sep = 2pt] (y) at (340:\R){};
    \Edges(u,x,v,y,u)
    \Edge(v)(u)
    \Edge(x)(y)        \node at (-0.5,-2) {$\bK$};
\end{tikzpicture}
\]
\end{observation}

\begin{proof}[Proof of Proposition \ref{nonedgeeqrel}]
Note that the claim trivially holds if $\Gamma$ is a tournament, so we continue under the assumption that $\Gamma$ is not a tournament.

It is an easy observation that $\not\sim$ is an equivalence relation if and only if $\Gamma$ does not contain configuration 
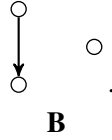
\[
\begin{tikzpicture}[scale=\scf,baseline=-17pt]
        \GraphInit[vstyle=Classic]
        \SetVertexCircle
        \SetupArcStyle
        \SetVertexNoLabel
   		\Vertex[x=0,y=-0.5]{y}
   		\Vertex[x=1,y=0]{z}
   		\Vertex[x=0,y=0.5]{x}
        \Edges(x,y)
        \node at (0.5,-1) {$\bB$};
\end{tikzpicture}\,. 
\]

 Next, we show that $\Gamma$ contains configuration $\bB$  if and only if it contains one of the following configurations:
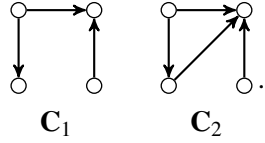
\[
\begin{tikzpicture}[scale=\scf,baseline=-15pt]
        \GraphInit[vstyle=Classic]
        \SetVertexCircle
        \SetupArcStyle
        \SetVertexNoLabel
   		\Vertex[x=0,y=-0.5]{y}
   		\Vertex[x=1,y=0.5]{z}
   		\Vertex[x=1,y=-0.5]{u}
   		\Vertex[x=0,y=0.5]{x}
        \Edge(x)(y)
        \Edge(x)(z)
        \Edge(u)(z)
        \node at (0.5,-1) {$\bC_1$};
\end{tikzpicture} \qquad
\begin{tikzpicture}[scale=\scf,baseline=-15pt]
        \GraphInit[vstyle=Classic]
        \SetVertexCircle
        \SetupArcStyle
        \SetVertexNoLabel
   		\Vertex[x=0,y=-0.5]{y}
   		\Vertex[x=1,y=0.5]{z}
   		\Vertex[x=1,y=-0.5]{u}
   		\Vertex[x=0,y=0.5]{x}
        \Edge(x)(y)
        \Edge(x)(z)
        \Edge(u)(z)
        \Edge(y)(z)
        \node at (0.5,-1) {$\bC_2$};
\end{tikzpicture}\,.         
\]
So suppose that $\Gamma$ contains configuration $\bB$. Let us fix in $\Gamma$  copies of $\bL_1$ and $\bB$:
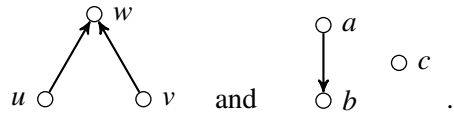
\[
\begin{tikzpicture}[scale=\scf,baseline=-14pt]
        \GraphInit[vstyle=Classic]
        \SetVertexCircle
        \SetupArcStyle
        \SetVertexLabel
		\SetVertexMath
   		\Vertex[a=0*120-30,d=0.75,L={v}]{v}
   		\Vertex[a=1*120-30,d=0.75,L={w}]{w}
   		\Vertex[a=2*120-30,d=0.75,L={u},Lpos=180]{u}
        \Edge(u)(w)
        \Edge(v)(w)
\end{tikzpicture} \quad \text{and}\qquad 
\begin{tikzpicture}[scale=\scf,baseline=-17pt]
        \GraphInit[vstyle=Classic]
        \SetVertexCircle
        \SetupArcStyle
        \SetVertexLabel
		\SetVertexMath
   		\Vertex[x=0,y=-0.5,L={b}]{b}
   		\Vertex[x=1,y=0,L={c}]{c}
   		\Vertex[x=0,y=0.5,L={a}]{a}
        \Edges(x,y)
\end{tikzpicture}\,.
\]
The existence of the former in $\Gamma$ is due to Corollary~\ref{confD}. Consider the local homomorphism $f\colon \{u,v\}\to \VGamma$ of $\Gamma$ given by $f:=\begin{psmallmatrix}  u & v \\ a & c\end{psmallmatrix}$. Since $\Gamma$ is homomorphism homogeneous, $f$ can be extended to a local homomorphism $\hat{f}\colon \{u,v,w\}\to \VGamma$. Observe that from $u\to w \gets v$ it follows $a=\hat{f}(u)\to \hat{f}(w)\gets \hat{f}(v)=c$, so we conclude that $\hat{f}(w)\notin\{a,b,c\}$, implying that the subgraph of $\Gamma$ induced by $\{a,b,c,\hat{f}(w)\}$ is isomorphic either to $\bC_1$ or to $\bC_2$:
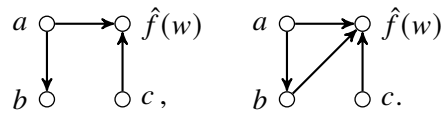
\[
\begin{tikzpicture}[scale=\scf,baseline=-17pt]
        \GraphInit[vstyle=Classic]
        \SetVertexCircle
        \SetupArcStyle
        \SetVertexLabel
		\SetVertexMath
   		\Vertex[x=0,y=-0.5,Lpos=-180,L={b}]{b}
   		\Vertex[x=1,y=0.5,L={\hat{f}(w)},Lpos=0]{f}
   		\Vertex[x=1,y=-0.5,Lpos=0,L={c\,,}]{c}
   		\Vertex[x=0,y=0.5,Lpos=-180,L={a}]{a}
        \Edge(a)(b)
        \Edge(a)(f)
        \Edge(c)(f)
\end{tikzpicture}\quad 
\begin{tikzpicture}[scale=\scf,baseline=-17pt]
        \GraphInit[vstyle=Classic]
        \SetVertexCircle
        \SetupArcStyle
        \SetVertexLabel
		\SetVertexMath
   		\Vertex[x=0,y=-0.5,Lpos=-180,L={b}]{b}
   		\Vertex[x=1,y=0.5,L={\hat{f}(w)},Lpos=0]{f}
   		\Vertex[x=1,y=-0.5,Lpos=0,L={c.}]{c}
   		\Vertex[x=0,y=0.5,Lpos=-180,L={a}]{a}
        \Edge(a)(b)
        \Edge(a)(f)
        \Edge(c)(f)
         \Edge(b)(f)
\end{tikzpicture} 
\]
The third possibility, where $\hat{f}(w)\to b$ cannot appear by Lemma~\ref{NoInducedChain}, since $c\to \hat{f}(w)\to b$, but $c\not\sim b$. 

For the proof of the other direction suppose that $\Gamma$ contains one of the following induced subgraphs
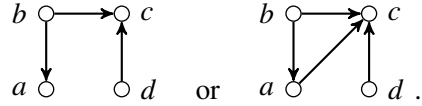
\[
\begin{tikzpicture}[scale=\scf, baseline=-18pt]
        \GraphInit[vstyle=Classic]
        \SetVertexCircle
        \SetupArcStyle
        \SetVertexLabel
		\SetVertexMath
   		\Vertex[x=0,y=-0.5,Lpos=-180,L={a}]{b}
   		\Vertex[x=1,y=0.5,L={c},Lpos=0]{f}
   		\Vertex[x=1,y=-0.5,Lpos=0,L={d}]{c}
   		\Vertex[x=0,y=0.5,Lpos=-180,L={b}]{a}
        \Edge(a)(b)
        \Edge(a)(f)
        \Edge(c)(f)
\end{tikzpicture}\quad \text{or} \quad 
\begin{tikzpicture}[scale=\scf,baseline=-18pt]
        \GraphInit[vstyle=Classic]
        \SetVertexCircle
        \SetupArcStyle
        \SetVertexLabel
		\SetVertexMath
   		\Vertex[x=0,y=-0.5,Lpos=-180,L={a}]{b}
   		\Vertex[x=1,y=0.5,L={c},Lpos=0]{f}
   		\Vertex[x=1,y=-0.5,Lpos=0,L={d}]{c}
   		\Vertex[x=0,y=0.5,Lpos=-180,L={b}]{a}
        \Edge(a)(b)
        \Edge(a)(f)
        \Edge(c)(f)
         \Edge(b)(f)
\end{tikzpicture}.         
\]
Then $\{a,b,d\}$ induces configuration $\bB$.

It is now clear that the task of showing that $\not\sim$ is an equivalence relation reduces to the check of the (non-)containment of configurations $\bC_1$ and $\bC_2$ in $\Gamma$.

Suppose that there are $a,b,c,d\in \VGamma$ such that
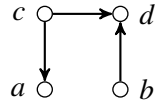
\[
\begin{tikzpicture}[scale=\scf,baseline=0pt]
 \GraphInit[vstyle=Classic]
        \GraphInit[vstyle=Classic]
        \SetVertexCircle
        \SetupArcStyle
        \SetVertexLabel
		\SetVertexMath
   		\Vertex[x=0,y=-0.5,Lpos=-180,L={a}]{b}
   		\Vertex[x=1,y=0.5,L={d},Lpos=0]{f}
   		\Vertex[x=1,y=-0.5,Lpos=0,L={b}]{c}
   		\Vertex[x=0,y=0.5,Lpos=-180,L={c}]{a}
        \Edge(a)(b)
        \Edge(a)(f)
        \Edge(c)(f)
\end{tikzpicture}
\]
 is an induced subgraph in $\Gamma$. Lemma~\ref{4conf} gives us the existence of $x,y,u,v\in \VGamma$ such that 
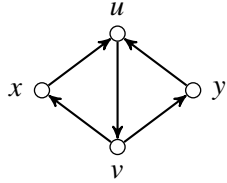
\[
\begin{tikzpicture}[scale=\scf,baseline=0pt]
        \GraphInit[vstyle=Classic]
        \SetVertexCircle
        \SetupArcStyle
        \SetVertexLabel
		\SetVertexMath
  		\Vertex[x=0,y=-0.75,Lpos=-90,L={v}]{y}
  		\Vertex[x=1,y=0,L={y}]{z}
  		\Vertex[x=-1,y=0,Lpos=180,L={x}]{u}
  		\Vertex[x=0,y=0.75,Lpos=90,L={u}]{x}
       \Edges(x,y,u,x)
       \Edges(y,z,x)
\end{tikzpicture}
\]
is an induced subgraph in $\Gamma$. Consider the local homomorphism $f\colon \{x,y,v\}\to \VGamma$ of $\Gamma$ given by $f:=\begin{psmallmatrix}  x & y & v \\ a & d & c\end{psmallmatrix}$. Since $\Gamma$ is homomorphism homogeneous, $f$ can be extended to a local homomorphism $\hat{f}\colon \{x,y,v,u\}\to \VGamma$. Then $b\to d=\hat{f}(y)\to \hat{f}(u)$. From Lemma~\ref{NoInducedChain} it follows that $b\sim\hat{f}(u)$. But, if $b\to\hat{f}(u)$, then since $\hat{f}(u)\to \hat{f}(v)=c$ we get $b\sim c$. Similarly, if $\hat{f}(u)\to b$, then, since $a=\hat{f}(x)\to \hat{f}(u)$, we get $a\sim b$. In both cases we arrive at a contradiction.\\
 Finally, suppose that there are $a,b,c,d\in \VGamma$ such that
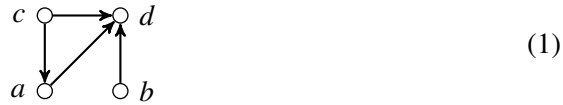
\begin{equation}\label{eq3}
    \begin{tikzpicture}[scale=\scf,baseline=0pt]
        \GraphInit[vstyle=Classic]
        \SetVertexCircle
        \SetupArcStyle
        \SetVertexLabel
		\SetVertexMath
   		\Vertex[x=0,y=-0.5,Lpos=-180,L={a}]{b}
   		\Vertex[x=1,y=0.5,L={d},Lpos=0]{f}
   		\Vertex[x=1,y=-0.5,Lpos=0,L={b}]{c}
   		\Vertex[x=0,y=0.5,Lpos=-180,L={c}]{a}
        \Edge(a)(b)
        \Edge(a)(f)
        \Edge(c)(f)
        \Edge(b)(f) 
    \end{tikzpicture}
\end{equation}
 is an induced subgraph in $\Gamma$. We proceed by studying the following two cases:
	\textbf{Case 1.} The core of $\Gamma$ is $C_3$. Then $\Gamma$ has no induced subgraph of the shape 
 	
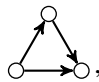
\[
 	\begin{tikzpicture}[scale=\scf,baseline=-8pt]
        \GraphInit[vstyle=Classic]
        \SetVertexCircle
        \SetupArcStyle
        \SetVertexNoLabel
		\SetVertexMath
   		\Vertex[a=0*120-30,d=0.5,L={v}]{v}
   		\Vertex[a=1*120-30,d=0.5,L={w}]{w}
   		\Vertex[a=2*120-30,d=0.5,L={u},Lpos=180]{u}
        \Edge(u)(w)
        \Edge(w)(v)
        \Edge(u)(v)
         \end{tikzpicture}\,,
    \]
 but the vertex set $\{a,c,d\}$ induces one such, leading to a contradiction.	

 \noindent\textbf{Case 2.} The core of $\Gamma$ is $S(2)$ or $T^{\infty}$. Then $\Gamma$ contains the following induced subgraph	
 
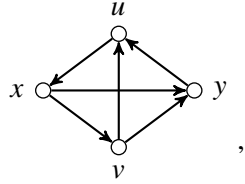
\[
    \begin{tikzpicture}[scale=\scf,baseline=-22pt]
        \GraphInit[vstyle=Classic]
        \SetVertexCircle
        \SetupArcStyle
        \SetVertexLabel
		\SetVertexMath
   		\Vertex[x=0,y=-0.75,Lpos=-90,L={v}]{y}
   		\Vertex[x=1,y=0,L={y}]{z}
   		\Vertex[x=-1,y=0,Lpos=180,L={x}]{u}
   		\Vertex[x=0,y=0.75,Lpos=90,L={u}]{x}
        \Edge(y)(x)
        \Edges(y,z,x,u,y)
        \Edge(u)(z)
    \end{tikzpicture},
  \]
since both $S(2)$ and $T^{\infty}$ contain $\bK$ (this is clear for $T^\infty$ and follows from Observation~\ref{S2K}  for $S(2)$). Consider the local homomorphism $f\colon \{v,y,x\}\to \VGamma$ of $\Gamma$ given by $f:=\begin{psmallmatrix}  v & y & x \\ a & d & c\end{psmallmatrix}$ (see \eqref{eq3}). Since $\Gamma$ is homomorphism homogeneous, $f$ can be extended to a local homomorphism  $\hat{f}\colon \{v,y,x,u\}\to \VGamma$. Then $b\to d=\hat{f}(y)\to \hat{f}(u)$. Again, from Lemma~\ref{NoInducedChain} it follows that $b\sim\hat{f}(u)$. But, if $b\to\hat{f}(u)$, then, since $\hat{f}(u)\to \hat{f}(x)=c$, we get $b\sim c$. Similarly, if $\hat{f}(u)\to b$, then since $a=\hat{f}(v)\to \hat{f}(u)$ we get $a\sim b$. In both cases we arrive at a contradiction.

 Thus $\Gamma$ contains neither $\bC_1$ nor $\bC_2$. Hence, $\not\sim$ is an equivalence relation.
\end{proof}

\begin{proposition}\label{ABorBA}
Let $\Gamma$ be a homomorphism homogeneous oriented graph that contains $C_3$. Let $A$ and $B$ be distinct equivalence classes of $\not\sim$. Then either $A\to B$ or $B\to A$ in $\Gamma$.	
\end{proposition}

\begin{proof}
Suppose the opposite. Then $A\times B\not\subseteq \EGamma$ and $B\times A\not\subseteq \EGamma$, so there exist $a_1,a_2\in A$, $b_1,b_2\in B$ such that $(a_1,b_1), (b_2,a_2)\not\in \EGamma$. It follows that $(b_1,a_1), (a_2,b_2)\in \EGamma$, and $a_1\neq a_2$ or $b_1\neq b_2$. Without loss of generality, suppose $a_1\neq a_2$. If $b_1=b_2$, then $a_2\to b_2=b_1\to a_1$ --- a contradiction with Lemma~\ref{NoInducedChain}. Hence $b_1\neq b_2$.

Since $a_1$ and $b_2$ are from different equivalence classes of $\not\sim$, it follows that $a_1\sim b_2$. If $a_1\to b_2$, then $b_1\to a_1\to b_2$. If $b_2\to a_1$, then $a_2\to b_2\to a_1$. In both cases we again arrive at a contradiction with Lemma~\ref{NoInducedChain}. 
\end{proof}

At this point the general structure of homomorphism homogeneous graphs containing $C_3$ is clear. Each of them can be obtained from a tournament using the following construction: 
\begin{definition}
    Let $\Gamma$ be a countable oriented graph, and let $S$ be a countable set. Let $f\colon S\to \VGamma$ be surjective. Then the oriented graph $\Gamma[f]$ is given by 
    \[
    \V(\Gamma[f]) = S, \text{and } \E(\Gamma[f]) = \{(s,t)\mid (f(s),f(t))\in \EGamma\}.
    \]
\end{definition}
\begin{remark}
    In the definition above, if $\Gamma$ is a tournament, then $\ker f$ is equal to the non-arc relation $\not\sim$ of $\Gamma[f]$. Roughly speaking, the graph $\Gamma[f]$ is obtained from $\Gamma$ by replacing each vertex $v$ of $\Gamma$ by $|f^{-1}(v)|$ independent copies of $v$. Moreover, if in $\Gamma$ a vertex $v$ has an arc to a vertex $w$, then in $\Gamma[f]$, every copy of $v$ is connected by an arc to every copy of $w$.     
    A similar construction appears in Cherlin's classification of homogeneous oriented graphs (see \cite[Section II]{Cher1998}). There, for any $n\le\omega$, the graph $\Gamma[I_n]$ is obtained from $\Gamma$ through replacing each vertex with an independent set of $n$ vertices. In our notation $\Gamma[I_n]$ is isomorphic to $\Gamma[f]$, where $f$ is a function whose kernel classes are all of equal size $n$. 
    Note that whenever $f\colon S\to V(\Gamma)$ and $g\colon T\to V(\Gamma)$ are surjective functions such that for all $v\in V(\Gamma)$ we have that $|f^{-1}(v)|=|g^{-1}(v)|$, then $\Gamma[f]\cong\Gamma[g]$. In other words, $\Gamma[f]$ is determined, up to isomorphism by the function $n_f\colon V(\Gamma)\to (\omega\cup\{\omega\})\setminus\{0\}$ given by $n_f\colon v\mapsto |f^{-1}(v)|$. 
    Abusing notation, for every function $\chi\colon V(\Gamma)\to(\omega\cup\{\omega\})\setminus\{0\}$ we denote by $\Gamma[\chi]$  any graph of the shape $\Gamma[f]$ for which $n_f=\chi$.
\end{remark}

\begin{example}
	Let $C_3$ be given by
	
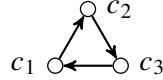
\[
	    \begin{tikzpicture}[scale=\scf,baseline=-8pt]
        \GraphInit[vstyle=Classic]
        \SetVertexCircle
        \SetupArcStyle
        \SetVertexLabel
 		\SetVertexMath
  		\Vertex[a=0*120-30,d=0.5,L={c_3}]{v}
   		\Vertex[a=1*120-30,d=0.5,L={c_2}]{w}
   		\Vertex[a=2*120-30,d=0.5,L={c_1},Lpos=180]{u}
        \Edges(w,v,u,w)
    \end{tikzpicture}
	\]
	Let $S:=\{s_1,s_2,s_3,s_4,s_5,s_6\}$, and let $f\colon S\to V(C_3)$ be given by 
	\[
		f:= \begin{pmatrix}
			s_1 & s_2 & s_3 & s_4 & s_5 & s_6\\
			c_1 & c_2 & c_2 & c_3 & c_3 & c_3 
		\end{pmatrix}.
	\] 
	Then $\Gamma[f]$ looks as follows:
	
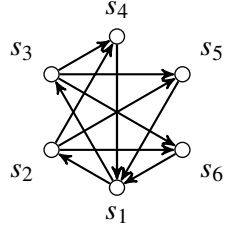
\[
	    \begin{tikzpicture}[scale=\scf,baseline=-8pt]
        \GraphInit[vstyle=Classic]
        \SetVertexCircle
        \SetupArcStyle
        \SetVertexLabel
 		\SetVertexMath
  		\Vertex[a=0*60-90,d=1,L={s_1},Lpos=0*60-90]{s1}
  		\Vertex[a=-1*60-90,d=1,L={s_2},Lpos=-1*60-90]{s2}
  		\Vertex[a=-2*60-90,d=1,L={s_3},Lpos=-2*60-90]{s3}
  		\Vertex[a=-3*60-90,d=1,L={s_4},Lpos=-3*60-90]{s4}
  		\Vertex[a=-4*60-90,d=1,L={s_5},Lpos=-4*60-90]{s5}
  		\Vertex[a=-5*60-90,d=1,L={s_6},Lpos=-5*60-90]{s6}
  		\Edges(s1,s2,s4,s1)
  		\Edges(s1,s3,s4,s1)
  		\Edges(s2,s5,s1)
  		\Edges(s2,s6,s1)
  		\Edges(s3,s5)
  		\Edges(s3,s6)
 		\end{tikzpicture}
	\]
	Moreover, we have 
	\[
		n_f = 
		\begin{pmatrix}
 			c_1 & c_2 & c_3\\
 			1 & 2 & 3
 		\end{pmatrix}.
	\]
\end{example}

\begin{proposition}\label{CyclicTransitive}
Let $\Gamma$ be a homomorphism homogeneous oriented graph that contains $C_3$. Then $\Gamma$ is isomorphic to $C_3[f]$, $S(2)[f]$, or $T^\infty[f]$, for some $f$.
\end{proposition}

\begin{proof}
  Let $\Gamma$ be a homomorphism homogeneous oriented graph that contains $C_3$, and let $T$ be its core. In particular, $T\in\{C_3, S(2), T^\infty\}$. Let $(v_i)_{i<\omega}$ be a transversal of $\VGamma/\not\sim$. Let $f\colon \VGamma\to \VGamma$ be the function that assigns to each $x\in \VGamma$ the unique $v_i$ such that $[x]_{\not\sim} = [v_i]_{\not\sim}$.  Then $f$ is an endomorphism of $\Gamma$ by Proposition~\ref{ABorBA}, and  $\im f$ induces a tournament $\Delta$ in $\Gamma$. 

\noindent\textbf{Claim 1.} $\Age (\Delta)=\Age (T)$. To see that this holds, recall that $T$ is the core of $\Gamma$, so $T$ and $\Gamma$ are homomorphism equivalent. On the other hand, $f\colon \Gamma\to \Delta$, and $\Delta \le \Gamma$, so $\Gamma$ and $\Delta$ are homomorphism equivalent, implying that $\Delta$ and $T$ are homomorphism equivalent, too. Since both $\Delta$ and $T$ are tournaments, it follows that $\Delta$ embeds into $T$ and $T$ embeds into $\Delta$. Hence, $\Age (\Delta)=\Age (T)$.

\noindent\textbf{Claim 2.} $\Delta$ is homogeneous. To prove this fact, let $g\colon A\to B$ be a local isomorphism of $\Delta$, and let $v_j\in \VDelta\setminus A$. 
Then $g$ is a local homomorphism of $\Gamma$. Since $\Gamma$ is homomorphism homogeneous, there exists $\hat{g}\in \End (\Gamma)$ that extends $g$. Let $k<\omega$ be such that $v_k=f(\hat{g}(v_j))$. Then $g'\colon A\cup \{v_j\}\to B\cup \{v_k\}$ defined by 
\[g'\colon x\mapsto \begin{cases} 
                           g(x), & \text{if }x\in A,\\
	                       v_k, & \text{if } x=v_j
\end{cases}
\]
is a one-point-extension of $g$ to $v_j$.
Hence, $\Delta$ is weakly homogeneous (see the remark after the proof), and thus homogeneous.
From the previous two observations and from \Fraisse's Theorem we conclude that $\Delta\cong T$.

Observe now that $\Gamma=\Delta[f]$.
Since $\Delta\cong T$, this finishes the proof.   
\end{proof}
\begin{remark}
    In this proof we used the term ``weakly homogeneous''. An oriented graph $\Gamma$ is called \emph{weakly homogeneous} if for every local isomorphism $f$ of $\Gamma$ with domain $A$ and for every finite superset $\hat{A}\supseteq A$ in $\VGamma$ there exists a local isomorphism $\hat{f}$ of $\Gamma$ with domain $\hat{A}$ that extends $f$. An easy back-and-forth argument shows that a countable oriented graph is weakly homogeneous if and only if it is homogeneous (see \cite[Lemma 6.1.4(b)]{Hod97}).
\end{remark}

At this point it remains to check, for which $f$ the graphs $T[f]$ (for $T\in\{C_3,S(2),T^\infty\}$)  are indeed homomorphism homogeneous. Towards this goal, let us analyse, what it means for $T[f]$ \emph{not} to be homomorphism homogeneous. In this case there must exist finite sets $A,B\subseteq V(T[f])$, a surjective function $h\colon A\to B$, and a vertex $v\in V(T[f])$, such that $h$ is a local homomorphism of $T[f]$ that cannot be extended to a local homomorphism with domain $A\cup\{v\}$. A quadruple $(A,B,h,v)$ with these properties is called a \emph{witness} (see \cite[Definition 2.4]{FelPecPec20}, \cite[Section 3.1.1]{MPPhD}). A witness, for which $|A|$ is minimal, is called a \emph{minimal witness}. In other words, if  $(A,B,h,v)$ is a minimal witness, then for all $(A',B',h',v')$, such that $|A'|<|A|$, $h'\colon A'\to B'$ is a surjective local homomorphism of $T[f]$, there is an extension of $h'$ to a local homomorphism of $T[f]$ with domain $A'\cup\{v'\}$. The following technical lemma characterizes possible minimal witnesses in oriented graphs of shape $T[f]$, where $T\in\{C_3,S(2),T^\infty\}$:
\begin{lemma}\label{lem:minwit}
	Let $T\in\{C_3,S(2),T^\infty\}$, let $S$ be a countable set,  and let $f\colon S\to V(T)$ be a surjective function, such that $T[f]$ is not homomorphism homogeneous. Let, further, $(A,B,h,v)$ be a minimal witness in $T[f]$. Then the following are true:
	\begin{enumerate}[label=(\arabic*), ref=(\arabic*)]
		\item\label{it1} $h$ is a bijection,
		\item\label{it2} $B$ induces a tournament in $T[f]$,
		\item\label{it3} $\forall a\in A\,:\, a\sim v$,
		\item\label{it4} $A$ does not induce a tournament in $T[f]$.
	\end{enumerate}
	Further, if $\widehat{B}$ is the tournament that is obtained from $B$ by adjoining a vertex $\hat{v}$ in such a way that 
	\[
		\tag{$\ast$}\label{eq:star} \forall a\in A\,:\, (a\to v)  \iff (h(a)\to \hat{v}),
	\]
	then $\widehat{B}\notin\Age(T)$. However, all proper oriented subgraphs of $\widehat{B}$ are in $\Age(T)$. 
\end{lemma}
\begin{proof}
	\noindent\textbf{About \ref{it1}.} Suppose that $h(a_1)=h(a_2)$, but $a_1\neq a_2$. Then $a_1\not\sim a_2$. Let $A':=A\setminus\{a_2\}$, $h':=h\restr_{A'}$. By minimality, $(A',B,h',v)$ is not a witness. Let $\hat{h}'$ be an extension of $h'$ to a local homomorphism of $T[f]$ with domain $A'\cup\{v\}$. Define an extension $\hat{h}\colon A\cup\{v\}\to B\cup\{\hat{h}'(v)\}$ of $h$ according to $\hat{h}(v) := \hat{h}'(v)$.  Since $a_1\not\sim a_2$, we have 
	\[
		(v\to a_2) \implies (v\to a_1) \implies \hat{h}(v) = \hat{h}'(v)\to\hat{h}'(a_1) = \hat{h}(a_2).
	\]
	Analogously we obtain that $a_2\to v$ implies $\hat{h}(a_2)\to \hat{h}(v)$. Thus, $\hat{h}$ is a local homomorphism, a contradiction. Consequently, $h$ is injective and hence bijective.
	
	\noindent\textbf{About \ref{it2}.} Suppose, there exist $b_1,b_2\in B$ such that $b_1\not\sim b_2$. Let $a_1\coloneqq h^{-1}(b_1)$, $a_2\coloneqq h^{-1}(b_2)$,  $A':=A\setminus \{a_2\}$,   $B':= B\setminus\{b_2\}$, and let $h'\colon A'\to B'$ be the appropriate restriction of $h$. Then $h'$ is a local homomorphism of $T[f]$. By the minimality assumption,  $(A',B',h',v)$ is not a witness. Let $\hat{h}'$ be a local homomorphism  of $T[f]$ that extends $h'$ to $A'\cup\{v\}$. Define an extension  $\hat{h}\colon A\cup\{v\}\to B\cup\{\hat{h}'(v)\}$ of $h$ according to  $\hat{h}(v):=\hat{h}'(v)$. Since  $a_1\not\sim a_2$, we have  $a_1\to v$ if and only if $a_2\to v$, and $v\to a_1$ if and only if $v\to a_2$. Since $b_2\not\sim b_1$, we also have $b_1\to \hat{h}'(v)$ if and only if $b_2\to\hat{h}'(v)$ and $\hat{h}'(v)\to b_1$ if and only if $\hat{h}'(v)\to b_2$. This is enough to conclude that $\hat{h}$ is indeed a homomorphism, a contradiction. It follows that $B$ induces a tournament in $T[f]$. 
	
	\noindent\textbf{About \ref{it3}.} Suppose on the contrary that for some $a\in A$ we have $a\not\sim v$. Then we may extend $h$ to $A\cup\{v\}$ by defining $\hat{h}(v):= h(a)$. Indeed, for all $a'\in A$ we have $a'\to v$ if and only if $a'\to a$, and $v\to a'$ if and only if $a\to a'$. This implies that $\hat{h}$ is a local homomorphism extending $h$, a contradiction. It follows that $a\sim v$, for all $a\in A$.
	
	\noindent\textbf{About \ref{it4}.} Suppose that $A$ induces a tournament in $T[f]$. Then $h$ induces an isomorphism. Let $f_A\colon A\to f[A]$ and $f_B\colon B\to f[B]$ be the appropriate restrictions of $f$. Note that both $f_A$ and $f_B$ are isomorphisms of tournaments (here we identify $A$, $B$, $f[A]$, $f[B]$ with their respective subtournaments in $T[f]$ and in $T$). Let $h'\colon f[A]\to f[B]$ be the unique mapping that makes the following diagram commutative:
	\[
	\begin{tikzcd}
		A \ar[r,"h"]\ar[d,"f_A"']& B\ar[d,"f_B"] \\
		f[A]\ar[r,dashed,"h'"] & f[B].
	\end{tikzcd}
	\]
	In other words, $h'= f_B\circ h\circ f_A^{-1}$. 
	Note now that $h'$ is a local isomorphism of $T$. Let $\hat{h}'$ be an extension of $h'$ to an automorphism of $T$. Let $v':=f(v)$, $\hat{v}':=\hat{h}'(v')$. Finally, let $\hat{v}$ be an arbitrary element of $f^{-1}(\hat{v}')$, and define an extension $\hat{h}\colon A\cup\{v\}\to B\cup\{\hat{v}\}$ of $h$ according to $\hat{h}\colon v\mapsto \hat{v}$. Then $\hat{h}$ is a local homomorphism of $T[f]$. Indeed, for $a\in A$ we may compute:
		\begin{align*}
			(a\to v) &\iff (f(a)\to f(v)) \iff (\hat{h}'(f(a))\to  \hat{h}'(f(v)))\\
			&\iff (\hat{h}'(f(a))\to\hat{h}'(v')) \\
			&\iff f^{-1} (h'(f(a)))\to f^{-1}(\hat{h}'(v')) = f^{-1}(\hat{v}').\tag{\dag}\label{eq:dag}
		\end{align*}
		Since $f(\hat{h}(a)) = f_B(h(a))= h'(f_A(a))= h'(f(a))$, we have that  $\hat{h}(a)\in f^{-1} (h'(f(a))$. Moreover,   $\hat{h}(v)=\hat{v}\in f^{-1}(\hat{v}')$, so we obtain, using Proposition~\ref{ABorBA},  that \eqref{eq:dag} is equivalent to $\hat{h}(a)\to\hat{h}(v)$. It follows that $(A,B,h,v)$ is not a witness after all --- a contradiction. We conclude that $A$ does not induce a tournament in $T[f]$. 
		
	\noindent\textbf{About $\bf\widehat{B}\notin\Age(T)$.} Suppose that $\widehat{B}\in\Age(T)$. Clearly, the subtournament of $T$ that is induced by $f[B]$ is isomorphic to $B$. For simplicity, let us assume that $B$ is actually a subtournament of $T$ (and that $f\restr_B$ is the identical embedding). Then, by the homogeneity of $T$, there exists an embedding $\iota\colon \widehat{B}\to T$, such that the following diagram commutes:
	\[
	\begin{tikzcd}
		T\\
		B \ar[u,hook,"="',"f\restr_B"] \ar[r,hook,"="]& \widehat{B}.\ar[ul,hook',"\iota"']
	\end{tikzcd} 
	\]
	Let $\hat{v}':=\iota(\hat{v})$ and $v'\in f^{-1}(\hat{v}')$. Define an extension $h'\colon A\cup \{v\}\to B\cup\{v'\}$ of $h$ according to $h'\colon v\mapsto v'$. We claim that $h'$ is a local homomorphism of $T[f]$. To see this, let $a\in A$. Then 
	\begin{align*}
		(a\to v) &\iff (h(a)\to \hat{v}) \iff (\iota(h(a)) \to \hat{v}')\\
		& \iff (f(h(a))\to\hat{v}') \iff h'(a)= h(a)\to v' = h'(v).
	\end{align*}
	Thus, $h'$ is indeed a local homomorphism --- a contradiction. It follows that $\widehat{B}\notin\Age(T)$. 
	
	\noindent\textbf{About the minimality of $\bf\widehat{B}$.} Suppose some proper subtournament $B'$ of $\widehat{B}$ is not in $\Age(T)$. Then $\hat{v}\in B'$ and we may define $A':=h^{-1}[B'\setminus\{\hat{v}\}]$, and $h':=h\restr_{A'}$. Clearly, then $(A',B',h',v)$ is a witness, a contradiction to the minimality assumption. Thus $\widehat{B}$ is indeed a minimal forbidden subtournament for $\Age(T)$.  
\end{proof}
\begin{proposition}\label{prop:back}
	The following are true:
	\begin{enumerate}[label=(\arabic*), ref=(\arabic*)]
		\item\label{rt1} $C_3[f]$ is homomorphism homogeneous, if and only if $f$ is a bijection;
		\item\label{rt2} $S(2)[f]$ is homomorphism homogeneous if and only if  
		\[
		\forall w\in V(S(2))\,:\, |f^{-1}(w)| \le 2;\tag{\ddag}\label{eq:ddag}
		\]
		\item\label{rt3} $T^\infty[f]$ is homomorphism homogeneous for all $f$ with countable domain.
	\end{enumerate}
\end{proposition}
\begin{proof}
	\noindent\textbf{About \ref{rt1}.} The ``if''-direction is clear. For the ``only if''-direction, suppose that $f$ is not bijective (and hence it is not injective). Let $a_1,a_2\in V(C_3[f])$ be distinct, such that $f(a_1)=f(a_2)$. Let $v\in V(C_3[f])$, such that $v\to\{a_1,a_2\}$. Let $b_1,b_2\in V(C_3[f])$, such that $b_1\to b_2$. Let $A:=\{a_1,a_2\}$, $B:=\{b_1,b_2\}$, $h\colon A\to B$ be given by $h\colon a_i\mapsto b_i$ ($i=1,2$), then $(A,B,h,v)$ is a witness. Indeed, if $\hat{h}$ is an extension of $h$ to $A\cup\{v\}$, then the image of $\hat{h}$ induces:
	
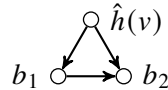
\[
	    \begin{tikzpicture}[scale=\scf,baseline=-8pt]
        \GraphInit[vstyle=Classic]
        \SetVertexCircle
        \SetupArcStyle
        \SetVertexLabel
 		\SetVertexMath
  		\Vertex[a=0*120-30,d=0.5,L={b_2}]{b2}
   		\Vertex[a=1*120-30,d=0.5,L={\hat{h}(v)}]{v}
   		\Vertex[a=2*120-30,d=0.5,L={b_1},Lpos=180]{b1}
		\Edge(v)(b1)
		\Edge(v)(b2)
		\Edge(b1)(b2)
    \end{tikzpicture}
	\] 
	which is not in the age of $C_3[f]$. It follows that $C_3[f]$ is not homomorphism homogeneous.
	
	\noindent\textbf{About \ref{rt2}.} Concerning the ``if''-part, suppose that \eqref{eq:ddag} holds but that $S(2)[f]$ is not homomorphism homogeneous. Then there has to exist a minimal witness $(A,B,h,v)$. By Lemma~\ref{lem:minwit}, $B$ induces a tournament, $h$ is bijective, $A$ does not induce a tournament, and $\widehat{B}$ is a minimal forbidden substructure for $\Age(S(2))$. It is well-known that $\Age(S(2))$ consists of all such finite tournaments that do not contain either of the following two configurations (see \cite[page 18]{Cher1998}):
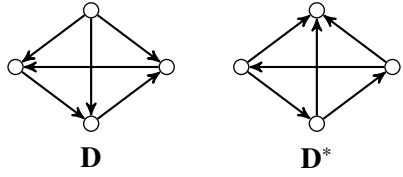
	 \[
    \begin{tikzpicture}[scale=\scf,baseline=-22pt]
        \GraphInit[vstyle=Classic]
        \SetVertexCircle
        \SetupArcStyle
        \SetVertexNoLabel
   		\Vertex[x=0,y=-0.75]{b}
   		\Vertex[x=1,y=0]{c}
   		\Vertex[x=-1,y=0]{a}
   		\Vertex[x=0,y=0.75]{d}
        \Edge(d)(a)
        \Edge(d)(b)
        \Edge(d)(c)
        \Edges(a,b,c,a)
        \node at (0,-1.2) {$\bD$};
    \end{tikzpicture}\qquad
    \begin{tikzpicture}[scale=\scf,baseline=-22pt]
        \GraphInit[vstyle=Classic]
        \SetVertexCircle
        \SetupArcStyle
        \SetVertexNoLabel
   		\Vertex[x=0,y=-0.75]{b}
   		\Vertex[x=1,y=0]{c}
   		\Vertex[x=-1,y=0]{a}
   		\Vertex[x=0,y=0.75]{d}
        \Edge(a)(d)
        \Edge(b)(d)
        \Edge(c)(d)
        \Edges(a,b,c,a)
        \node at (0,-1.2) {$\bD^\ast$};
    \end{tikzpicture}\label{forbS2}
  \]
It follows that $\widehat{B}$ is either isomorphic to $\bD$ or to $\bD^\ast$. Let us handle the case  $\widehat{B}\cong \bD$ (the case $\widehat{B}\cong\bD^\ast$ goes analogously): In this case $B$ must be isomorphic either to $C_3$ or to the three-vertex transitive tournament. 

In case that $B\cong C_3$, since $h$ is bijective, $A$ has three vertices and is obtained from $C_3$ by removing at least one arc. However, removing any arcs from $C_3$ leaves us with a configuration not contained in $S(2)[f]$. Thus $B$ must be transitive. Again,  removing more than one arc from $B$ leads to a configuration not contained in $S(2)[f]$. It follows that $A\cup\{v\}$ induces one of the following three configurations:
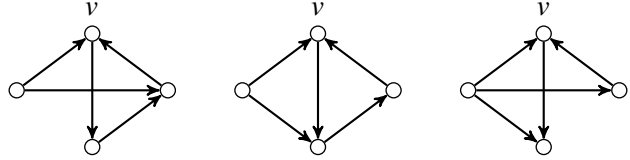
\[
    \begin{tikzpicture}[scale=\scf,baseline=-22pt]
        \GraphInit[vstyle=Classic]
        \SetVertexCircle
        \SetupArcStyle
        \SetVertexMath
        \SetVertexNoLabel
   		\Vertex[x=0,y=-0.75]{b}
   		\Vertex[x=1,y=0]{c}
   		\Vertex[x=-1,y=0]{a}
		\SetVertexLabel
   		\Vertex[x=0,y=0.75,Lpos=90,L={v}]{d}
        \Edge(a)(c)
        \Edge(b)(c)
        \Edge(d)(b)
        \Edge(a)(d)
        \Edge(c)(d)
    \end{tikzpicture}\qquad
    \begin{tikzpicture}[scale=\scf,baseline=-22pt]
        \GraphInit[vstyle=Classic]
        \SetVertexCircle
        \SetupArcStyle
        \SetVertexMath
        \SetVertexNoLabel
   		\Vertex[x=0,y=-0.75]{b}
   		\Vertex[x=1,y=0]{c}
   		\Vertex[x=-1,y=0]{a}
		\SetVertexLabel
   		\Vertex[x=0,y=0.75,Lpos=90,L={v}]{d}
        \Edge(a)(b)
        \Edge(b)(c)
        \Edge(d)(b)
        \Edge(a)(d)
        \Edge(c)(d)
    \end{tikzpicture}\qquad
        \begin{tikzpicture}[scale=\scf,baseline=-22pt]
        \GraphInit[vstyle=Classic]
        \SetVertexCircle
        \SetupArcStyle
        \SetVertexMath
        \SetVertexNoLabel
   		\Vertex[x=0,y=-0.75]{b}
   		\Vertex[x=1,y=0]{c}
   		\Vertex[x=-1,y=0]{a}
		\SetVertexLabel
   		\Vertex[x=0,y=0.75,Lpos=90,L={v}]{d}
        \Edge(a)(b)
        \Edge(a)(c)
        \Edge(d)(b)
        \Edge(a)(d)
        \Edge(c)(d)
    \end{tikzpicture}
    \]
However, neither of these configurations is contained in $S(2)[f]$. 
But that means that $(A,B,h,v)$ cannot be a witness. In other words, our assumption was wrong and  $S(2)[f]$ is homomorphism homogeneous. 

Concerning the ``only if''-part, suppose that for some $w\in V(S(2))$ we have that $|f^{-1}(w)|\ge 3$. Let $a_1$, $a_2$, and $a_3$ be distinct vertices from $f^{-1}(w)$. Let $v\in V(S(2)[f])$, be such that $v\to\{a_1,a_2,a_3\}$, and let $b_1,b_2,b_3\in V(S(2)[f])$ be such that $b_1\to b_2\to b_3 \to b_1$. Define $A:=\{a_1,a_2,a_3\}$, $B:=\{b_1,b_2,b_3\}$, and $h\colon A\to B$ according to $h(a_i)=b_i$ ($i=1,2,3$). Then $(A,B,h,v)$ is a witness. Clearly, $h$ is a local homomorphism. Moreover, if $h$ could be extended to a local homomorphism  $\hat{h}$ with domain $A\cup\{v\}$  then the image of $\hat{h}$ would necessarily induce a tournament isomorphic to $\bD$, a contradiction. It follows that $|f^{-1}(w)|\le 2$ for all $w\in V(S(2))$.

	\noindent\textbf{About \ref{rt3}.} Suppose that $T^\infty[f]$ is not homomorphism homogeneous. Then it should have a minimal witness $(A,B,h,v)$. But then, by Lemma~\ref{lem:minwit}, $\widehat{B}$ is a tournament not in the age of $T^\infty$, a contradiction. Hence, $T^\infty[f]$ is homomorphism homogeneous.
\end{proof}

Combining Propositions~\ref{acyclicHH},  \ref{CyclicTransitive}, and \ref{prop:back} we finally obtain: 
\begin{theorem}\label{MainHH}
  Let $\Gamma$ be a countable homomorphism homogeneous weakly connected oriented graph. Then $\Gamma$ is one of the following oriented graphs:
    \begin{enumerate}[label=(\arabic*), ref=(\arabic*)]
    \item\label{class1} $I_1$, 
	\item\label{class2} $(\bQ, <)$,
	\item\label{class3} a tree with no minimal elements such that no finite subset of vertices has a maximal lower bound,
	\item\label{class4} a dual tree with no maximal elements such that no finite subset of vertices has a minimal upper  bound,
	\item\label{class5} a poset such that: \begin{itemize}
 \item every finite subset of vertices is bounded from above and from below,
 \item no finite subset of vertices has a maximal lower bound or a minimal upper bound,
 \item no $X_4$-set has a midpoint,	
 \end{itemize}
  \item\label{class6} any extension of the countable universal homogeneous  strict poset,
  \item \label{class7}$C_3$,
  \item \label{class8}$S(2)[f]$ for a surjective function $f$ with kernel classes of size at most $2$,
  \item \label{class9} $T^\infty[f]$ for a surjective function $f$ with a countable domain.
    \end{enumerate}
\end{theorem}

\section{Weakly connected polymorphism homogeneous oriented graphs}\label{sec3}

We turn now our attention to polymorphism homogeneous oriented graphs, an important subclass of the class of homomorphism homogeneous oriented graphs. 
\begin{definition}\label{def:PH}
	Let $\Gamma$ be an oriented graph. Then the \emph{$n$-th direct power} $\Gamma^n$ of $\Gamma$ is defined by:
	\begin{align*}
		V(\Gamma^n) &:=V(\Gamma)^n,\\
		E(\Gamma^n) &:= \{((v_1,\dots,v_n),(w_1,\dots,w_n))\mid \forall i\in\{1,\dots,n\}\,:\, (v_i,w_i) \in E(\Gamma)\}.
	\end{align*}
	Moreover, $\Gamma$ is called \emph{polymorphism homogeneous} if the $n$-th power of $\Gamma$ is homomorphism homogeneous, for all positive integers $n$.
\end{definition}	

As in the previous section, we will restrict our study to weakly connected oriented graphs, and the disconnected case will be treated separately. Our starting point will be the case of countable polymorphism homogeneous tournaments, that were fully classified in \cite[Theorem 3.10]{FelPecPec20}. 
An immediate consequence of this theorem is:
\begin{corollary}
Let $T$ be a countable tournament. Then, $T$ is polymorphism homogeneous if and only if it is isomorphic to either $I_1$, ${C}_3$, or $(\bQ, <)$.
\end{corollary}

Moving on to non-tournaments we distinguish once more between the cyclic and acyclic among them motivated by the fact that
a countable homomorphism homogeneous oriented graph is transitive if and only if it is acyclic. This allows us again to see countable acyclic non-tournaments as countable strict partially ordered sets. The latter have already been classified with respect to polymorphism homogeneity in~\cite{PecPec15}.

\begin{proposition}[{\cite[Theorem 6.29]{PecPec15}}]
The only countable polymorphism homogeneous acyclic non-tournaments are extensions of the countable universal homogeneous strict poset.
\end{proposition}

  It remains to tackle the class of cyclic non-tournaments (recall that by Lemma~\ref{alwaysC3}, every cyclic polymorphism homogeneous oriented graph, if there are any such,  contains $C_3$). In the following we show that this class contains no polymorphism homogeneous oriented graph:
\begin{lemma}\label{KnonPoly}
    If an oriented graph $\Gamma$ contains configuration $\bK$ (see Observation~\ref{S2K}), then $\Gamma$ is not polymorphism homogeneous.
\end{lemma}
\begin{proof}
    Let $a,b,c,d\in \VGamma$ be such that they induce in $\Gamma$ a subgraph isomorphic to $\bK$:
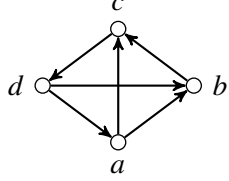
\[
    \begin{tikzpicture}[scale=\scf,baseline=-22pt]
        \GraphInit[vstyle=Classic]
        \SetVertexCircle
        \SetupArcStyle
        \SetVertexLabel
		\SetVertexMath
        \Vertex[x=1,y=-0.75,Lpos=-90,Ldist=0,L={a}]{a}
        \Vertex[x=2,y=0,Lpos=0,Ldist=0,L={b}]{b}
        \Vertex[x=1,y=0.75,Lpos=90,Ldist=0,L={c}]{c}
        \Vertex[x=0,y=0,Lpos=180,Ldist=0,L={d}]{d}
        \Edges(a,b,c,d,a)			
        \Edges(d,b)
        \Edges(a,c)
    \end{tikzpicture}
\]
Observe that  $(b,d)\to (c,a)\to (d,b)$ in $\Gamma^2$, because $b\to c\to d$ and $d\to a\to b$ in $\Gamma$, and  that $(b,d)\not\sim(d,b)$ in $\Gamma^2$ (since $b\not\to d$, and hence $(b,d)\not\to(d,b)$, and $(d,b)\not\to(b,d)$). By Lemma \ref{NoInducedChain} we obtain that $\Gamma^2$ is not homomorphism homogeneous. Consequently, $\Gamma$ is not polymorphism homogeneous. 
\end{proof}
\begin{proposition}
    If $\Gamma$ is a weakly connected polymorphism homogeneous oriented graph that contains an oriented cycle, then $\Gamma\cong C_3$.
\end{proposition}
\begin{proof}
  Proposition~\ref{CyclicTransitive} provides us with our only potential candidates for $\Gamma$. These are $\Gamma\cong C_3$, $\Gamma\cong S(2)[f]$, or $\Gamma\cong T^\infty[f]$ for suitable $f$.

If $\Gamma\cong C_3$, then nothing needs to be proved.
 
 Since both $S(2)$ and $T^\infty$, contain configuration $\bK$ (see Observation~\ref{S2K}), so does $\Gamma$. Thus neither $S(2)[f]$ nor $T^\infty[f]$ is polymorphism homogeneous, by Lemma~\ref{KnonPoly}.
\end{proof}
The results of this section are summed up in the following theorem:
\begin{theorem}\label{mainphcon}
Let $\Gamma$ be a finite or countably infinite polymorphism homogeneous weakly connected oriented graph. Then $\Gamma$ is one of the following oriented graphs:
    \begin{enumerate}
    \item $I_1$, 
    \item $C_3$,
	\item $(\bQ, <)$,
  \item an extension of the countable universal homogeneous  strict poset.
    \end{enumerate}    
\end{theorem}

\section{Disconnected homomorphism homogeneous oriented graphs}\label{sec4}
We continue our study of homomorphism homogeneous oriented graphs by considering the disconnected case. Our first observation is:
\begin{proposition}\label{ll_wcc}
If $\Gamma$ is a homomorphism homogeneous disconnected oriented graph, then all of its weakly connected components are tournaments.
\end{proposition}
\begin{proof}
  Assume the opposite, that there exists a non-edge within some weakly connected component. Thus there exist $x, y \in \VGamma$ such that
  $x \sim^\ast y$, but $x \not\sim y$. Because $\Gamma$ is disconnected, there exists  $z \in \VGamma$ that is not in the same weakly connected component as $x$ and $y$.

  Consider the map $f = \left( \begin{smallmatrix}
                               x & y \\
                               z & y
                             \end{smallmatrix} \right)$.
  Clearly it is a local homomorphism of $\Gamma$. Due to $\Gamma$'s homomorphism homogeneity, there exists $\hat{f} \in \End(\Gamma)$ which extends $f$.
  Therefore, $z = f(x) = \hat{f}(x) \sim^\ast \hat{f}(y) = f(y) = y$, which is a contradiction.
\end{proof}
Next, let us examine weakly connected components of homomorphism homogeneous oriented graphs in more detail.

\begin{lemma}\label{wccHH}
Weakly connected components of a homomorphism homogeneous disconnected oriented graph $\Gamma$ are also homomorphism homogeneous.
\end{lemma}
\begin{proof}
  Let  $C$ be  a weakly connected component of $\Gamma$, and let $f$ be a non-trivial local homomorphism of $C$.
  Then $f$ is also a local homomorphism of $\Gamma$. Thus, there exists $\hat{f} \in \End(\Gamma)$ which extends $f$. Clearly, $\hat{f}[C]\subseteq C$. Thus $\hat{f}\restr_C$ is an endomorphism of $C$ that extends $f$. 
\end{proof}

\begin{proposition}\label{wcc=Ages}
If $\Gamma$ is a homomorphism homogeneous disconnected oriented graph, then all of its weakly connected components have the same age.
\end{proposition}
\begin{proof}
  Let $C_1$ and $C_2$ be two different weakly connected components of $\Gamma$. Take any $x \in C_1$ and $y \in C_2$. Consider the following local homomorphism
  $f = \left( \begin{smallmatrix}
         x \\
         y
       \end{smallmatrix} \right)$. By the homomorphism homogeneity of $\Gamma$, there exists $\hat{f} \in \End(\Gamma)$ which extends $f$. 
       Clearly, $\hat{f}[C_1]\subseteq C_2$. Since $C_1$ is a tournament, $\hat{f}\restr_{C_1}$ is an embedding. In other words, $C_1$ embeds into $C_2$.
        Analogously it can be shown that $C_2$ embeds into $C_1$. Consequently, $\Age(C_1) = \Age(C_2)$.
\end{proof}
\begin{remark}
    In the following, if an oriented graph $\Gamma$ has exactly $k$ weakly connected components each of which induces a subgraph isomorphic to a given oriented graph $\Delta$, then we denote $\Gamma$ also by $k\cdot \Delta$. 
\end{remark}
\begin{theorem}\label{wccsametour}
Let $\Gamma$ be a countable disconnected oriented graph. Then $\Gamma$ is homomorphism homogeneous if and only if all of its weakly connected components are isomorphic to the same
homogeneous tournament. In particular, it is isomorphic to one of the following oriented graphs:
    \begin{enumerate}
        \item $k\cdot I_1$,
        \item $k\cdot C_3$,
        \item $k\cdot (\bQ,<)$,
        \item $k\cdot S(2)$,
        \item $k\cdot T^{\infty}$,
    \end{enumerate}
    for some $1<k\le\omega$.
\end{theorem}
\begin{proof}
  Assume, at first, that $\Gamma$ is homomorphism homogeneous. Proposition~\ref{ll_wcc} together with Lemma~\ref{wccHH} imply that all weakly connected components of $\Gamma$
  are homomorphism homogeneous tournaments. 
  Recall that a countable tournament is homomorphism homogeneous if and only if it is homogeneous. 
  Now, combining Proposition~\ref{wcc=Ages} with \Fraisse's Theorem,  we come to the conclusion that all weakly connected components of $\Gamma$ are isomorphic to one and the same homogeneous tournament.

  Consider now the opposite direction, assuming that all of $\Gamma$'s weakly connected components are isomorphic to the same homogeneous tournament. Take any local homomorphism $f$ of $\Gamma$.
  Let $C_1,\dots,C_m$ be all connected components of $\Gamma$ such that $\dom (f)\cap C_i\neq \emptyset$  for each $i \in \{1, 2, \dots, m\}$. Each time denote $\dom (f)\cap  C_i$ by $A_i$.
  In particular, 
  \[
  \dom(f) = A_1 \dotcup A_2 \dotcup \dots \dotcup A_m.
  \]
  Define $B_i := f[A_i]$, for all $i \in \{1, 2, \dots, m\}$. 
  Further define $f_i := f\restr_{A_i} \colon A_i \to B_i$, for all $i \in \{1, 2, \dots, m\}$. Notice how each $B_i$ is fully contained within some weakly connected component $D_i$ of $\Gamma$. Now  $C_i$ and $D_i$ are isomorphic homogeneous tournaments for each $i$, so we may use \Fraisse's Theorem in order to conclude that there exists an isomorphism $\hat{f}_i \colon C_i\to  D_i$ which extends $f_i$.

  Finally, we define the following extension of $f$ on $\Gamma$:
  \[
    \hat{f}(x) := 
      \begin{cases}
        \hat{f}_i(x), & \mbox{if } x \in C_i,\, i \in \{1, 2, \dots, m\} \\
        x, & \mbox{otherwise}.
      \end{cases}
 \]
  Clearly, $\hat{f}$ is an endomorphism. Thus $\Gamma$ is homomorphism homogeneous.
\end{proof}

\section{Disconnected polymorphism homogeneous oriented graphs}\label{sec5}
Now that we know all homomorphism homogeneous disconnected oriented graphs, let us see which of them are polymorphism homogeneous.

\begin{proposition}\label{disconn wcc PH}
All weakly connected components of a disconnected polymorphism homogeneous oriented graph are pairwise isomorphic homogeneous polymorphism homogeneous tournaments.
\end{proposition}
\begin{proof}
    Let $\Gamma$ be a disconnected polymorphism homogeneous oriented graph with weakly connected components $T_0,T_1,\dots,T_m$. By Theorem~\ref{wccsametour} all these components are pairwise isomorphic homogeneous tournaments. For each $i\in\{1,\dots,m\}$, let $\varphi_i\colon T_i\to T_0$ be an isomorphism. Fix a positive natural number $n$ and consider a local homomorphism $f$ of $T^n_{0}$. Then $f$ is also a local homomorphism of $\Gamma^n$. Since $\Gamma^n$ is homomorphism homogeneous, there exists $\bar{f}\in\End(\Gamma^n)$ that extends $f$. Let 
    \[
    \varphi\colon \VGamma\to \VGamma\qquad x\mapsto\begin{cases}
      \varphi_i(x) & \text{if } x\in T_i,\, i\in\{1,\dots,m\},\\
      x & \text{if } x\in T_0.
    \end{cases}
    \]
    Obviously, $\varphi\in\End(\Gamma)$. Consequently, 
    \[
    \varphi^n\colon \V(\Gamma^n)\to \V(\Gamma^n)\qquad (x_1,\dots,x_n)\mapsto (\varphi(x_1),\dots,\varphi(x_n))
    \]
    is an endomorphism of $\Gamma^n$. Finally, we define 
    $\hat{f}\coloneqq \varphi^n \circ\bar{f}$. Note that $\hat{f}$ also extends $f$, but its image is completely contained in $T_0^n$. Thus $\hat{f}\restr_{T_0^n}$ is an endomorphism of $T_0^n$ that extends $f$. Hence, $T_0$ is polymorphism homogeneous. 
\end{proof}

\begin{lemma}\label{kQnotPH}
For all $1<k\leq \omega$ we have that $k \cdot (\bQ, <)$ is not polymorphism homogeneous.
\end{lemma}

\begin{proof} Fix  $k \geq 2$, and assume the opposite, that $k\cdot (\bQ,<)$ is polymorphism homogeneous. Thus there exist two different weakly connected components $C_1$ and $C_2$,
and vertices $a, b, c \in C_1$ and $d \in C_2$ such that $a \to b \to c$. Observe that then $(a,b)\to (c,c)\gets (b,a)$, and consider the local homomorphism of $(k\cdot (\bQ, <))^2$:
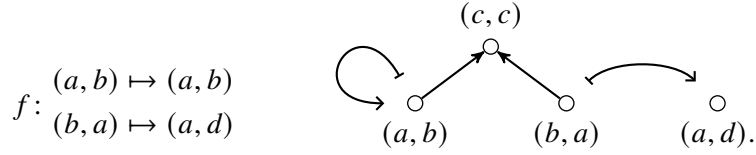
\[
  f\colon \begin{aligned}
      (a,b) &\mapsto  (a,b) \\
      (b,a) &\mapsto  (a,d) 
  \end{aligned}\qquad 
 \begin{tikzpicture}[scale=\scf,baseline=-4pt]
        \GraphInit[vstyle=Classic]
        \SetVertexCircle
        \SetupArcStyle
        \SetVertexLabel
		\SetVertexMath
        \Vertex[x=0,y=0,Lpos=-90,Ldist=0,L={(a,b)}]{x1}
        \Vertex[x=1,y=0.75,Lpos=90,Ldist=0,L={(c,c)}]{x2}
        \Vertex[x=2,y=0,Lpos=-90,Ldist=0,L={(b,a)}]{x3}
        \Vertex[x=4,y=0,Lpos=-90,Ldist=0,L={(a,d).}]{x5}
        \Edges(x1,x2)
        \Edges(x3,x2)
        \tikzset{>={Straight Barb[scale=0.8]}}
		\tikzset{EdgeStyle/.style = {|->,thick}}
        \tikzset{|/.tip={Bar[width=.8ex]}}
        \tikzset{every loop/.style={min distance=15mm,looseness=10}}
        \draw[EdgeStyle,shorten >=0.3cm,shorten <=0.3cm] (x1) to [out=120, in=180,loop] ();
        \draw[EdgeStyle,shorten >=0.3cm,shorten <=0.3cm](x3) edge [bend left=45] (x5);
  \end{tikzpicture}
\]
Since $k\cdot (\bQ, <)$ is polymorphism homogeneous, it follows that $(k\cdot (\bQ, <))^2$ is homomorphism homogeneous, so $f$ can be extended to $\bar{f}\in \End((k\cdot (\bQ, <))^2)$. But then $(a,b)=\bar{f}(a,b)\to \bar{f}(c,c)\gets \bar{f}(b,a)=(a,d)$, but this cannot be satisfied, since $b$ and $d$ belong to distinct weakly connected components, and so we arrive at a contradiction.

This implies that $(k \cdot (\bQ, <))^2$ is not homomorphism homogeneous nor is $k \cdot (\bQ, <)$ polymorphism homogeneous.
\end{proof}

\begin{theorem}\label{mainphdiscon}
The only countable polymorphism homogeneous disconnected oriented graphs, up to isomorphism, are:
\begin{enumerate}
  \item $k \cdot I_1$ and  
  \item $k \cdot C_3$,
\end{enumerate}
for any $1 < k \le \omega$.
\end{theorem}

\begin{proof}
    By Proposition~\ref{disconn wcc PH} and \cite[Theorem 3.10]{FelPecPec20} the only candidates for weakly connected components are $I_1$, $C_3$, and $(\bQ,<)$. Lemma~\ref{kQnotPH} rules out $(\bQ,<)$. As for $I_1$ and $C_3$ we note that
  for any $n > 0$
  \[\left.\begin{array}{rcl}
            \left(k \cdot {C}_3 \right)^n &\cong& (k^n \cdot 3^{n-1}) \cdot {C}_3 \\
             & & \\
            (k \cdot I_1)^n  &\cong& k^n \cdot I_1
          \end{array}\right\}\;\; \text{are homomorphism homogeneous}.\]
  Thus $k \cdot C_3$ and  $k \cdot I_1$ are indeed polymorphism homogeneous, for any $1 < k\le \omega$.
\end{proof}

\section{Discussion}\label{discussions}
	
	Let us continue with a comparison of our results with Cherlin's classification of homogeneous oriented graphs. The nature of this discussion is going to be less formal than the rest of the paper. 
	
	A first observation is that a countable disconnected oriented graph is homomorphism homogeneous if and only if it is homogeneous (see Theorem~\ref{wccsametour} and \cite[Catalog 2]{Cher1998}).
	
	In the classification of countable weakly connected HH oriented graphs (Theorem~\ref{MainHH}), only a couple of homogeneous oriented graphs appear,  namely $I_1$ (class~\ref{class1}), $\bQ$ (class~\ref{class2}), $\mathcal{P}$ (the countable universal homogeneous strict poset, class~\ref{class6}), $\bQ[I_n]$ (class \ref{class6}), $C_3$ (class~\ref{class7}), $S(2)$, $S(2)[I_2]$ (class~\ref{class8}), $T^\infty[I_n]$ (class~\ref{class9}), where $1\le n\le\omega$. This accounts for only countably many of the homogeneous oriented graphs from Cherlin's classification. In particular, none of the  Henson digraphs is homomorphism homogeneous. On the other hand, no proper member of classes~\ref{class3}, \ref{class4}, and \ref{class5} is homogeneous.
	
	From the homogeneous HH connected oriented graphs mentioned above only $S(2)$, $S(2)[I_2]$, and $T^\infty[I_n]$ are \emph{not} polymorphism homogeneous (see Theorem~\ref{mainphcon}), and in the disconnected case, only $I_k$ and $k\cdot C_3$ \emph{are} polymorphism homogeneous (see Theorem~\ref{mainphdiscon}).
	
	An important difference between the classification of homogeneous oriented graphs and the classification of HH oriented graphs lies in the importance (or nonimportance) of ages. Namely, by \Fraisse's theorem there is a one-to-one correspondence between amalgamation classes of finite oriented graphs and isomorphism classes of countable homogeneous oriented graphs. This is vastly different for HH oriented graphs. Here we only have that two HH oriented  graphs of the same age are homomorphism equivalent. In fact,  two countable HH oriented graphs are homomorphism equivalent if and only if their homogeneous cores are isomorphic. It follows that there are just five homomorphism equivalence classes of countable HH oriented graphs (see Corollary~\ref{poss_cores}). Table~\ref{tab1}  gives an overview of the number of distinct ages and of pairwise non-isomorphic HH oriented graphs in each class from Theorem~\ref{MainHH}.
	\begin{table}[t]
	\begin{tabular}{@{}lrr@{}} \toprule
	& \# of ages & maximum \# of isomorphism classes per age\\ \midrule
	class \ref{class1}  & 1 & 1 \\
	class \ref{class2} & 1 & 1 \\
	class \ref{class3} & $\aleph_0$ & $\mathfrak{c}$ \\
	class \ref{class4} & $\aleph_0$ & $\mathfrak{c}$\\
	class \ref{class5} & $\mathfrak{c}$ & $\mathfrak{c}$\\
	class \ref{class6} & $\mathfrak{c}$ & $\mathfrak{c}$\\
	class \ref{class7} & 1 & 1 \\
	class \ref{class8} & $\aleph_0$ & $\mathfrak{c}$\\
	class \ref{class9} & $\mathfrak{c}$ & $\mathfrak{c}$\\\bottomrule\addlinespace[1ex]
	\end{tabular}
	\caption{Some details concerning the oriented graphs from Theorem~\ref{MainHH}}\label{tab1}
	\end{table}
	In the following we give additional explanations concerning the data in Table~\ref{tab1}. 
	
	\subsection*{About class~\ref{class3}:} The age of a tree from class~\ref{class3} consists of finite forests. It is known that the class of finite forests is well-quasi-ordered (wqo), when equipped with the embedding quasi-order (forests are $N$-free posets and finite $N$-free posets are wqo, see  \cite[Theorem 3.19]{Pou24},  \cite[Theorem 4]{Dam90}). It follows that every age of forests is determined by a finite set of forbidden subforests. Hence there are at most countably many ages of forests. To see that there are exactly $\aleph_0$ different ages of HH trees,   recall that the \emph{width} of a forest $(F,<)$ is the supremum of the lengths of all antichains in $F$. Clearly, if $F'$ is a forest that embeds into $F$, then the width of $F'$ is smaller or equal to the width of $F$. Now, for every $n\in\omega\setminus\{0\}$ we define $T_n$ to be the ordinal sum of $\bQ$ with $n$ disjoint copies of $\bQ$ (formally, $T_n = \bQ\oplus I_n[\bQ]$). Then $T_n$ has width $n$ and hence each element of $\Age(T_n)$ has width at most $n$. Moreover, some elements of $\Age(T_n)$ have width equal to $n$. It is not hard to see that $T_n$ belongs to class~\ref{class3}. Moreover, all the $T_n$ have distinct ages.
	
	Next we show that there are continuum many pairwise non-isomorphic countable HH trees whose age consists of all finite forests. For this 
	we need  to recall the notion of ramification orders: A \emph{ramification point} in a forest $(F,<)$ is an initial segment $X$ of $F$ that is linearly ordered. The \emph{ramification order} of $X$ is defined to be the number of equivalence classes of the equivalence relation $\asymp_X$ on the set $Y$ of strict upper bounds of $X$ in $F$, given by:
	\[ a\asymp_X b  :\iff \exists y\in Y\,:\, y\le a,b.\]
	Clearly, the set of ramification orders of a tree is an isomorphism invariant. Let us consider now the  tree $\mathfrak{B}$ that consists of all finite words over the alphabet $\omega$, and $w_1$ is below $w_2$ whenever $w_1$ is a prefix of $w_2$. For an arbitrary sequence $\overline{m}=(m_n)_{n<\omega}$ of positive integers, consider the subtree $\mathfrak{B}_{\overline{m}}$ induced by all   those words $w=a_0\dots a_{|w|-1}$ such that for all $2k< |w|$ we have that the $a_{2k}\in\{0,\dots, m_k-1\}$, and let $T$ be the tree obtained from $\mathfrak{B}_{\overline{m}}$ by replacing every element with a copy of $\bQ$. Then  $T$ is in class~\ref{class3}, and  the set of ramification orders of $T$ is equal to $M:=\{m_n\mid n<\omega\}\cup\{1,\omega\}$.  Moreover, the age of $T$ is equal to the class of all finite forests. This accounts for continuum many pairwise non-isomorphic HH trees, all of the same age. 
	
	Class~\ref{class4} is dual to class~\ref{class3} and does not require any further comment.
	\subsection*{About class~\ref{class5}} Each poset from class~\ref{class5} is the (not necessarily disjoint) union of a tree from  class~\ref{class3} and a dual tree from  class~\ref{class4}. Indeed, if $(P,<)$ is from class~\ref{class5}, then we may define $T:=\{x\in P\mid x{\downarrow} \text{ is a chain}\}$ and $D:=\{y\in P\mid y{\uparrow}\text{ is a chain}\}$. Then $T$ induces a tree from class~\ref{class3} and $D$ induces a dual tree from class~\ref{class4} in $P$. If there was a point $m\in P$ that is neither in $T$, nor in $D$, then $m$ would be the midpoint of some $X_4$-set. 
	
	Note that a finite poset is in the age of some poset from class~\ref{class5} if and only if the poset $\bX_5$ (see the picture on page~\pageref{X5}) does not embed into it. Let us denote the class of finite $\bX_5$-free posets by $\mathcal{C}$. With the same argument as above, every member of $\cC$ is the union of a finite forest with a finite dual forest. Given, that finite forests and finite dual forests are wqo with respect to embeddings it comes perhaps as a small surprise that $\cC$ is not. To see this, consider the antichain $(\bA_n)_{n<\omega}$ of $\bX_5$-free posets given in Figure~\ref{fig:An}. In order to obtain an uncountable family of ages of HH structures of class~\ref{class5}, we proceed with the following steps:
	\begin{figure}[t]
		\[
    	\begin{tikzpicture}[scale=\scf,baseline=0]
        	\GraphInit[vstyle=Classic]
        	\SetVertexCircle		
			\SetVertexLabel
	   		\Vertex[x=0,y=1,Lpos=125,L={$b_1$}]{b1}
   			\Vertex[x=2,y=1,Lpos=125,L={$b_2$}]{b2}
			\SetVertexNone
   			\Vertex[x=3.4,y=-0.20,Lpos=125,L={$x$}]{x}
   			\Vertex[x=5.5,y=0,Lpos=125,L={$y$}]{y}
			\SetVertexCircle
			\SetVertexLabel
			\Vertex[x=6,y=1,Lpos=125,L={$b_{n+2}$}]{bn2}
			\Vertex[x=8,y=1,Lpos=125,L={$b_{n+3}$}]{bn3}
   			\Vertex[x=-1,y=-1,Lpos=-125,L={$a_1$}]{a1}
   			\Vertex[x=1,y=-1,Lpos=-125,L={$a_2$}]{a2}
   			\Vertex[x=3,y=-1,Lpos=-125,L={$a_3$}]{a3}
   			\Vertex[x=7,y=-1,Lpos=-125,L={$a_{n+3}$}]{an3}
        	\path (x) -- node[auto=false]{\ldots} (y);
			\Edges(a1,b1,a2,b2,a3,x)
			\Edges(y,bn2,an3,bn3)
			\Edge(a1)(bn3)   		
    \end{tikzpicture}		
		\]
		\caption{The posets $\bA_n$}\label{fig:An}
	\end{figure}
	\begin{enumerate}
		\item For every subset $\sA\subseteq\omega$ we define $\bA_\sA$ to be the disjoint sum of all $\bA_n$ for which $n\in\sA$.
		\item  Define $\bP_\sA:=\{\bot\}\oplus \bA_\sA\oplus\{\top\}$, and denote its age by $\cC_\sA$.
		\item Note that $\bA_n\in\cC_\sA$ if and only if $n\in\sA$.
		\item Consider the poset $\bP_\sA[\bQ]$ that is obtained from $\bP_\sA$ through substituting each of its elements by a copy of $\bQ$, and denote its age by $\cC_\sA[\bQ]$.
		\item Note that $\bP_\sA[\bQ]$ is homomorphism homogeneous and of class~\ref{class5}.
		\item Observe that for every $1\le n<\omega$ we have that $\bA_n\in\cC_\sA[\bQ]$ if and only if $n\in \sA$.
	\end{enumerate}
	Thus we  constructed continuum many distinct ages of countable HH structures of class~\ref{class5}.
	
	It remains to show that for some age $\cC$ of $\bX_5$-free posets there exists continuum many pairwise non-isomorphic countable HH posets from class~\ref{class5} of  age $\cC$.  To this end, observe that if $T$ is in class~\ref{class3}, then the ordinal sum $\widehat{T}:=T\oplus\bQ$ is in class~\ref{class5}. Moreover, for any two trees $T_1$, $T_2$ from class~\ref{class3},  $T_1\not\cong T_2$ implies  $\widehat{T}_1\not\cong\widehat{T}_2$, and $\Age(T_1)=\Age(T_2)$ implies $\Age(\widehat{T}_1)= \Age(\widehat{T}_2)$.  So the continuum many pairwise non-isomorphic trees $\mathfrak{C}_{\overline{m}}$ constructed above give rise to continuum many pairwise non-isomorphic posets $\widehat{\mathfrak{C}}_{\overline{m}}$ of class \ref{class5}, all of the same age.
	\subsection*{About class~\ref{class6}} 
	According to  \cite[Proposition 4(a)]{CamLoc10}, a countable poset is of class~\ref{class6} if and only if 
	\begin{enumerate}[label=(\roman*), ref=(\roman*)]
		\item\label{p1} every finite subset of vertices is bounded from above and from below,
		\item\label{p2} no finite subset of vertices has a maximal lower bound or a minimal upper bound,	
		 \item every $X_4$-set has a midpoint.
	\end{enumerate}	
	It follows that the intersection of classes~\ref{class6} and \ref{class5} consists of all countable $X_4$-free posets that satisfy \ref{p1} and \ref{p2}. Note that all the class \ref{class5}  posets that were constructed in the previous subsection are $X_4$-free, and hence also in class~\ref{class6}.
	
	\subsection*{About class~\ref{class8}}

  The age of $S(2)$ consists of all finite local orders (recall that a tournament $T$ is called a \emph{local order} if for every $v\in V(T)$ the sets $\{x\mid x\to v\}$ and $\{y\mid v\to y\}$ induce transitive subtournaments in $T$). The class of finite local orders is wqo with respect to embeddings. This essentially follows from a result by Cameron that establishes a one-to-one correspondence (up to isomorphism) between  the class of finite local orders with a distinguished point and the set $\{0,1\}^+$ of non-empty finite words over the   alphabet $\{0,1\}$ (see \cite[Proposition 6.1]{Cam81}). It turns out that this correspondence actually defines an essentially surjective functor from the category of non-empty words over $\{0,1\}$  with subword embeddings to the category of finite non-empty local orders with tournament-embeddings.  In particular, it defines a quasi-order homomorphism from $\{0,1\}^+$ with the subword order to the finite local orders with the embedding quasi-order. By Higman's Lemma (see \cite[Theorem 4.3]{Hig52}), we have that $\{0,1\}^+$ with the subword ordering is well-partially-ordered. It follows that finite local orders with embedding quasi-order are wqo. Let us quickly describe Cameron's correspondence on the geometric model of $S(2)$ (see page~\pageref{s2def}). 
  
   Fix a vertex $\top$ of $S(2)$ (without loss of generality, we may assume that $\top$ is the point on the $y$-axis with coordinates $(0,1)$). Given a word $w=a_1\dots a_n\in\{0,1\}^+$, choose vertices  $\{p_1,\dots,p_n\}$ of $S(2)$ such that $\{p_1,\dots,p_n\}\to\top$ and such that $p_1\to p_2\to\dots\to p_n$ (this is possible because the vertices of $S(2)$ on the semi-circle to the right of  $\top$ induce a subtournament isomorphic to $\bQ$).  Next, for every $i$ with $a_i=1$ exchange $p_i$ by a vertex of $S(2)$ very close to the antipode of $p_i$ (we can go arbitrarily close to this antipode because of the density of $V(S(2))$ in the set of points of the circle). Define a tournament $T_w$ with $V(T_w)=\{1,\dots,n\}$ and with $i\to j$ if $p_i\to p_j$ in $S(2)$. It is important to note that $T_w$ does not depend on the initial choice of the  points $p_1,\dots,p_n$. Note that if $\iota\colon w_1\to w_2$ is a subword-embedding, then the same function $\iota$ embeds $T_{w_1}$ into $T_{w_2}$. Thus the assignment $w\mapsto T_w$ is indeed functorial. The construction for the case  $w=0101$ is illustrated in Figure~\ref{const1}.  
	\begin{figure}[ht]
	    \[
	\begin{tikzpicture}[scale=0.8\scf]
	\node at (-1.5,2) {Step 1};    
	\def\R{1.5}
    \tikzset{EdgeStyle/.style = {->}}
    \tikzset{>=stealth'}		
    \draw[->,line width=0.5] (-1.3*\R,0) -- (1.3*\R,0) node[right] {$x$};
    \draw[->,line width=0.5] (0,-1.3*\R) -- (0,1.4*\R) node[right] {$y$};
    \draw[line width = 0.5] (0,0) circle (\R);
    \node[shape=circle,draw,fill = black,minimum size = 2pt, inner sep = 2pt] (p4) at (67.5:\R){};
    \node (l4) at (67.5:\R+0.25*\R)  {$p_4$};
    \node[shape=circle,draw,fill = white,minimum size = 2pt, inner sep = 2pt] (p3) at (22.5:\R){};
    \node (l3) at (22.5:\R+0.25*\R)  {$p_3$};
    \node[shape=circle,draw,fill = black,minimum size = 2pt, inner sep = 2pt] (p2) at (-22.5:\R){};
    \node (l2) at (-22.5:\R+0.25*\R)  {$p_2$};
    \node[shape=circle,draw,fill = white,minimum size = 2pt, inner sep = 2pt] (p1) at (-67.5:\R){};
    \node (l1) at (-67.5:\R+0.25*\R)  {$p_1$};
    \node (top) at (95:\R+0.15*\R)  {$\top$};
\end{tikzpicture} \quad
	\begin{tikzpicture}[scale=0.8\scf]
	\node at (-1.5,2) {Step 2};    
    \def\R{1.5}
    \tikzset{EdgeStyle/.style = {->}}
    \tikzset{>=stealth'}		
    \draw[->,line width=0.5] (-1.3*\R,0) -- (1.3*\R,0) node[right] {$x$};
    \draw[->,line width=0.5] (0,-1.3*\R) -- (0,1.4*\R) node[right] {$y$};
    \draw[line width = 0.5] (0,0) circle (\R);
    \node (q4) at (67.5:\R){};
    \node[shape=circle,draw,fill = white,minimum size = 2pt, inner sep = 2pt] (p4) at (67.5+180:\R){};
    \node (l4) at (67.5+180:\R+0.25*\R)  {$p_4$};
    \node[shape=circle,draw,fill = white,minimum size = 2pt, inner sep = 2pt] (p3) at (22.5:\R){};
    \node (l3) at (22.5:\R+0.25*\R)  {$p_3$};
    \node (q2) at (-22.5:\R){};
	\node (l2) at (-22.5+180:\R+0.25*\R)  {$p_2$};
    \node[shape=circle,draw,fill = white,minimum size = 2pt, inner sep = 2pt] (p2) at (-22.5+180:\R){};
    \node[shape=circle,draw,fill = white,minimum size = 2pt, inner sep = 2pt] (p1) at (-67.5:\R){};
    \node (l1) at (-67.5:\R+0.25*\R)  {$p_1$};
    \node (top) at (95:\R+0.15*\R)  {$\top$};
    \draw[dotted] (q2)--(p2);
    \draw[dotted] (q4)--(p4);
\end{tikzpicture}
\quad
	\begin{tikzpicture}[scale=0.8\scf]
	\node at (-1.5,2) {Step 3};    
    \def\R{1.5}
    \tikzset{EdgeStyle/.style = {->}}
    \tikzset{>=stealth'}		
    \node (q4) at (67.5:\R){};
    \node[shape=circle,draw,fill = white,minimum size = 2pt, inner sep = 2pt] (p4) at (67.5+180:\R){};
    \node (l4) at (67.5+180:\R+0.25*\R)  {$4$};
    \node[shape=circle,draw,fill = white,minimum size = 2pt, inner sep = 2pt] (p3) at (22.5:\R){};
    \node (l3) at (22.5:\R+0.25*\R)  {$3$};
    \node (q2) at (-22.5:\R){};
	\node (l2) at (-22.5+180:\R+0.25*\R)  {$2$};
    \node[shape=circle,draw,fill = white,minimum size = 2pt, inner sep = 2pt] (p2) at (-22.5+180:\R){};
    \node[shape=circle,draw,fill = white,minimum size = 2pt, inner sep = 2pt] (p1) at (-67.5:\R){};
    \node (l1) at (-67.5:\R+0.25*\R)  {$1$};
    \Edges(p1,p3,p2,p4,p1)
    \Edge(p4)(p3)
    \Edge(p2)(p1)
\end{tikzpicture}
\]
\caption{From the word $w=0101$ to the tournament $T_w$}\label{const1}
\end{figure}
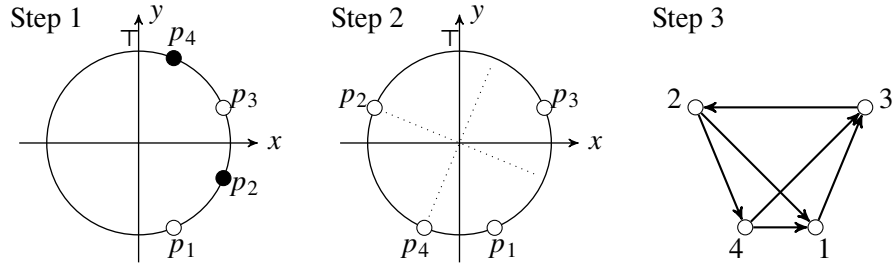
It is not hard to see that this construction is reversible: Starting from a finite subtournament $T$  of $S(2)$, choose a point $\top$ on the circle such that neither $\top$ nor its antipode is a vertex of $T$ (we may assume without loss of generality that $\top$ is on the $y$-axis with coordinates $(0,1)$). The diameter through $\top$ partitions the vertices of $T$ into two subsets $V_0$ and $V_1$ (it can happen that one of these sets is empty), where $V_0$ consists of all vertices on the semicircle right of $\top$ and $V_1$ of all vertices on the semicircle left of $\top$. Replace the vertices in $V_1$ by vertices of $S(2)$ very close to their respective antipodes. This moves them onto the semicircle to the right of $\top$ and defines a natural linear order on  $V_1\cup V_2$. Next the vertices from $V_1$ get assigned the letter $1$ and those from $V_0$ the letter $0$. This defines the desired word $w_T$. It remains to observe that  $T_{w_T}\cong T$. An example of this process is given in Figure~\ref{const2}.
	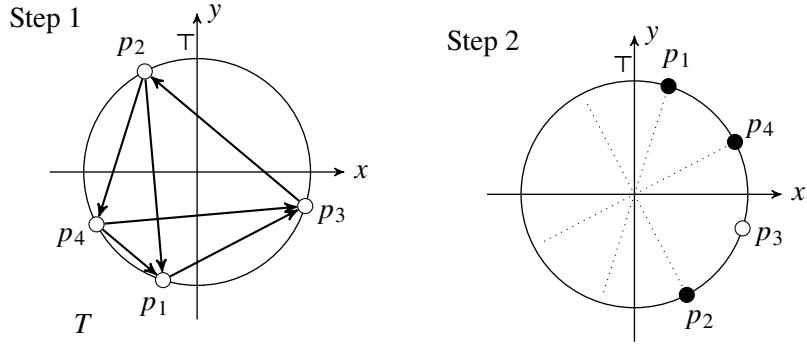
\begin{figure}[ht]
\[	\begin{tikzpicture}[scale=\scf]
	\node at (-1.5,-2) {$T$};	\node at (-2,2) {Step 1};    
    \def\R{1.5}
    \def\a{-40}
    \tikzset{EdgeStyle/.style = {->}}
    \tikzset{>=stealth'}		
    \draw[->,line width=0.5] (-1.3*\R,0) -- (1.3*\R,0) node[right] {$x$};
    \draw[->,line width=0.5] (0,-1.3*\R) -- (0,1.4*\R) node[right] {$y$};
    \draw[line width = 0.5] (0,0) circle (\R);
    \node (q4) at (67.5+\a:\R){};
    \node[shape=circle,draw,fill = white,minimum size = 2pt, inner sep = 2pt] (p4) at (67.5+180+\a:\R){};
    \node (l4) at (67.5+180+\a:\R+0.25*\R)  {$p_4$};
    \node[shape=circle,draw,fill = white,minimum size = 2pt, inner sep = 2pt] (p3) at (22.5+\a:\R){};
    \node (l3) at (22.5+\a:\R+0.25*\R)  {$p_3$};
    \node (q2) at (-22.5+\a:\R){};
	\node (l2) at (-22.5+180+\a:\R+0.25*\R)  {$p_2$};
    \node[shape=circle,draw,fill = white,minimum size = 2pt, inner sep = 2pt] (p2) at (-22.5+180+\a:\R){};
    \node[shape=circle,draw,fill = white,minimum size = 2pt, inner sep = 2pt] (p1) at (-67.5+\a:\R){};
    \node (q1) at (-67.5+180+\a:\R){};
    \node (l1) at (-67.5+\a:\R+0.25*\R)  {$p_1$};
    \node (top) at (95:\R+0.15*\R)  {$\top$};
    \Edges(p1,p3,p2,p4,p1)
    \Edge(p4)(p3)
    \Edge(p2)(p1)
\end{tikzpicture}\qquad
\begin{tikzpicture}[scale=\scf]
	\node at (-2,2) {Step 2};    
    \def\R{1.5}
    \def\a{-40}
    \tikzset{EdgeStyle/.style = {->}}
    \tikzset{>=stealth'}		
    \draw[->,line width=0.5] (-1.3*\R,0) -- (1.3*\R,0) node[right] {$x$};
    \draw[->,line width=0.5] (0,-1.3*\R) -- (0,1.4*\R) node[right] {$y$};
    \draw[line width = 0.5] (0,0) circle (\R);
    \node (q4) at (67.5+180+\a:\R){};
    \node[shape=circle,draw,fill = black,minimum size = 2pt, inner sep = 2pt] (p4) at (67.5+\a:\R){};
    \node (l4) at (67.5+\a:\R+0.25*\R)  {$p_4$};
    \node[shape=circle,draw,fill = white,minimum size = 2pt, inner sep = 2pt] (p3) at (22.5+\a:\R){};
    \node (l3) at (22.5+\a:\R+0.25*\R)  {$p_3$};
    \node (q2) at (-22.5+180+\a:\R){};
	\node (l2) at (-22.5+\a:\R+0.25*\R)  {$p_2$};
    \node[shape=circle,draw,fill = black,minimum size = 2pt, inner sep = 2pt] (p2) at (-22.5+\a:\R){};
    \node[shape=circle,draw,fill = black,minimum size = 2pt, inner sep = 2pt] (p1) at (-67.5+180+\a:\R){};
    \node (q1) at (-67.5+\a:\R){};
    \node (l1) at (-67.5+180+\a:\R+0.25*\R)  {$p_1$};
    \node (top) at (95:\R+0.15*\R)  {$\top$};
    \draw[dotted] (q2)--(p2);
    \draw[dotted] (q4)--(p4);
	\draw[dotted] (q1)--(p1);
\end{tikzpicture}\]
\caption{From from a finite tournament $T<S(2)$ to the word $w_T=1011$}\label{const2}
\end{figure}
	In this example $V_1=\{p_1,p_2,p_4\}$, $V_0=\{p_3\}$. The resulting order on $V_0\cup V_1$ is  $p_2<p_3<p_4<p_1$ and the resulting word is $w_T=1011$. 
	
	Next we claim that  $\Age(S(2)[I_2])$ is wqo with respect to embeddings, too. For this we need to augment the above described construction by another level of complexity. Instead of using words over the alphabet $\{0,1\}$,  we are going to use the partially ordered alphabet $\Sigma:=\{0,1\}\times\{1,2\}$, where $(i_1,j_1)<(i_2,j_2)$ if and only if $i_1=i_2$ and $(j_1<j_2)$. Again, by Higman's Lemma the set of finite words over  $\Sigma$ is wqo by the word-embedding-quasiorder (this time word-embeddings need to respect the partial order on the alphabet in the sense that the image of every letter must be greater or equal than this letter). The process of coming from a word $w=(a_1,m_1)\dots(a_n,m_n)$ to an oriented graph $\Gamma_w$ goes in two steps: First we take the word $w_1:=a_1\dots a_n$ and construct the tournament $T_{w_1}$ with vertex set $\{1,\dots,n\}$, just as described above. Next we define a function $\chi_w\colon\{1,\dots,n\}\to\{1,2\}$ according to $\chi_w\colon i\mapsto m_i$.
	Finally, $\Gamma_w:= T_w[\chi_w]$. Just as before, the mapping $w\mapsto\Gamma_w$ is functorial (we need to take care that the construction $T_w\mapsto T_w[\chi_w]$ is functorial, but this is not a  problem).  Again it is not hard, starting from an induced oriented subgraph $\Gamma$ of $S(2)[I_2]$, to construct a word $w_\Gamma\in\Sigma^+$ such that $\Gamma_{w_\Gamma}\cong\Gamma$. Instead of describing this process in detail, we illustrate it with  an example, see Figure~\ref{const3}.
	\begin{figure}[ht]
\[	\begin{tikzpicture}[scale=0.8\scf]
	\node at (-1.5,-2) {$\Gamma$};	\node at (-2,2) {Step 1};    
    \def\R{1.5}
    \def\a{-40}
    \tikzset{EdgeStyle/.style = {->}}
    \tikzset{>=stealth'}		
    \draw[->,line width=0.5] (-1.7*\R,0) -- (1.7*\R,0) node[right] {$x$};
    \draw[->,line width=0.5] (0,-1.7*\R) -- (0,1.7*\R) node[right] {$y$};
    \draw[line width = 0.5] (0,0) circle (\R);
    \draw[line width = 0.5] (0,0) circle (1.4*\R);
    \node (q4) at (67.5+\a:\R){};
    \node[shape=circle,draw,fill = white,minimum size = 2pt, inner sep = 2pt] (p4) at (67.5+180+\a:\R){};
    \node[shape=circle,draw,fill = white,minimum size = 2pt, inner sep = 2pt] (p42) at (67.5+180+\a:1.4*\R){};
    \node (l42) at (67.5+180+\a:1.4*\R+0.3*\R)  {$p_{4,2}$};
    \node[inner sep=0pt] (l41) at (67.5+180-20+\a:\R+0.8*\R)  {$p_{4,1}$};
    \draw[dashed] (l41)--(p4);
    \node[shape=circle,draw,fill = white,minimum size = 2pt, inner sep = 2pt] (p3) at (22.5+\a:\R){};
    \node (l3) at (22.5+\a:\R+0.25*\R)  {$p_3$};
    \node (q2) at (-22.5+\a:\R){};
	\node (l2) at (-22.5+180+\a:\R+0.25*\R)  {$p_2$};
    \node[shape=circle,draw,fill = white,minimum size = 2pt, inner sep = 2pt] (p2) at (-22.5+180+\a:\R){};
    \node[shape=circle,draw,fill = white,minimum size = 2pt, inner sep = 2pt] (p1) at (-67.5+\a:\R){};
    \node (q1) at (-67.5+180+\a:\R){};
    \node (l1) at (-67.5+\a:\R+0.25*\R)  {$p_1$};
    \node (top) at (95:1.4*\R+0.15*\R)  {$\top$};
    \Edges(p1,p3,p2,p4,p1)
    \Edges(p2,p42,p1)
    \Edge(p4)(p3)
    \Edge(p42)(p3)
    \Edge(p2)(p1)
\end{tikzpicture}
	\begin{tikzpicture}[scale=0.8\scf]
	\node at (-1.5,-2) {$\Gamma/\raisebox{-0.3ex}{$\!\not\sim$}$};	\node at (-1.4,2.3) {Step 2};    
    \def\R{1.5}
    \def\a{-40}
    \tikzset{EdgeStyle/.style = {->}}
    \tikzset{>=stealth'}		
    \draw[->,line width=0.5] (-1.3*\R,0) -- (1.3*\R,0) node[right] {$x$};
    \draw[->,line width=0.5] (0,-1.3*\R) -- (0,1.4*\R) node[right] {$y$};
    \draw[line width = 0.5] (0,0) circle (\R);
    \node (q4) at (67.5+\a:\R){};
    \node[shape=circle,draw,fill = white,minimum size = 2pt, inner sep = 1pt] (p4) at (67.5+180+\a:\R){$\scriptstyle 2$};
    \node (l4) at (67.5+180+\a:\R+0.3*\R)  {$p_4$};
    \node[shape=circle,draw,fill = white,minimum size = 2pt, inner sep = 1pt] (p3) at (22.5+\a:\R){$\scriptstyle 1$};
    \node (l3) at (22.5+\a:\R+0.3*\R)  {$p_3$};
    \node (q2) at (-22.5+\a:\R){};
	\node (l2) at (-22.5+180+\a:\R+0.3*\R)  {$p_2$};
    \node[shape=circle,draw,fill = white,minimum size = 2pt, inner sep = 1pt] (p2) at (-22.5+180+\a:\R){$\scriptstyle 1$};
    \node[shape=circle,draw,fill = white,minimum size = 2pt, inner sep = 1pt] (p1) at (-67.5+\a:\R){$\scriptstyle 1$};
    \node (q1) at (-67.5+180+\a:\R){};
    \node (l1) at (-67.5+\a:\R+0.3*\R)  {$p_1$};
    \node (top) at (95:\R+0.15*\R)  {$\top$};
    \Edges(p1,p3,p2,p4,p1)
    \Edge(p4)(p3)
    \Edge(p2)(p1)
\end{tikzpicture}
\begin{tikzpicture}[scale=0.8\scf]
	\node at (-1.4,2) {Step 3};    
    \def\R{1.5}
    \def\a{-40}
    \tikzset{EdgeStyle/.style = {->}}
    \tikzset{>=stealth'}		
    \draw[->,line width=0.5] (-1.3*\R,0) -- (1.3*\R,0) node[right] {$x$};
    \draw[->,line width=0.5] (0,-1.3*\R) -- (0,1.4*\R) node[right] {$y$};
    \draw[line width = 0.5] (0,0) circle (\R);
    \node (q4) at (67.5+180+\a:\R){};
    \node[shape=circle,draw,fill = lightgray,minimum size = 2pt, inner sep = 1pt] (p4) at (67.5+\a:\R){$\scriptstyle 2$};
    \node (l4) at (67.5+\a:\R+0.3*\R)  {$p_4$};
    \node[shape=circle,draw,fill = white,minimum size = 2pt, inner sep = 1pt] (p3) at (22.5+\a:\R){$\scriptstyle 1$};
    \node (l3) at (22.5+\a:\R+0.3*\R)  {$p_3$};
    \node (q2) at (-22.5+180+\a:\R){};
	\node (l2) at (-22.5+\a:\R+0.3*\R)  {$p_2$};
    \node[shape=circle,draw,fill = lightgray,minimum size = 2pt, inner sep = 1pt] (p2) at (-22.5+\a:\R){$\scriptstyle 1$};
    \node[shape=circle,draw,fill = lightgray,minimum size = 2pt, inner sep = 1pt] (p1) at (-67.5+180+\a:\R){$\scriptstyle 1$};
    \node (q1) at (-67.5+\a:\R){};
    \node (l1) at (-67.5+180+\a:\R+0.3*\R)  {$p_1$};
    \node (top) at (95:\R+0.15*\R)  {$\top$};
    \draw[dotted] (q2)--(p2);
    \draw[dotted] (q4)--(p4);
	\draw[dotted] (q1)--(p1);
\end{tikzpicture}\]
\caption{From  a finite graph $\Gamma<S(2)[I_2]$ to the word $w_\Gamma=(1,1)(0,1)(1,2)(1,1)$}\label{const3}
\end{figure}
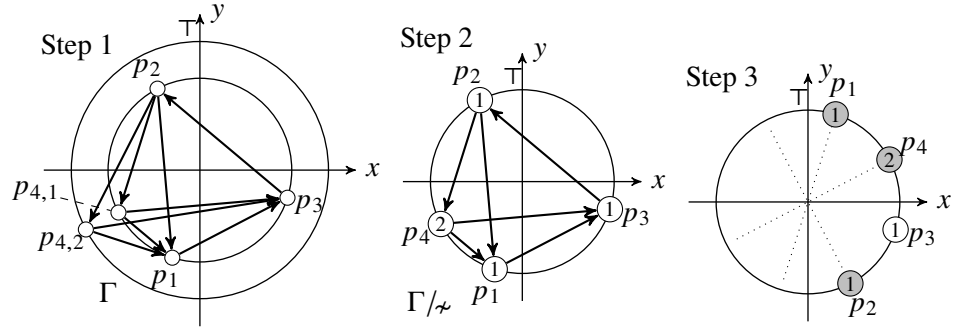
	
	Thus we conclude that $\Age(S(2)[I_2])$ is wqo. 
	This implies that there are at most $\aleph_0$  different ages of HH oriented graphs from class~\ref{class8}. It remains to show that the number of such ages is equal to $\aleph_0$. This follows from the observation that whenever  $\chi_1,\chi_2\colon V(S(2))\to\{1,2\}$ with  $|\chi_1^{-1}(2)| \neq |\chi_2^{-1}(2)|$, then  $\Age(S(2)[\chi_1])\neq\Age(S(2)[\chi_2])$.
	
	Next let us show that there are continuum many pairwise non-isomorphic HH oriented graphs in class~\ref{class8} whose age is equal to $\Age(S(2)[I_2])$. Let $(L,<)$ be a countable linear order.  Consider again the geometric model of $S(2)$ described above with the distinguished vertex $\top$ with the coordinates $(0,1)$ on the $y$-axis. Let $V_1\subset V(S(2))$ such that $V_1$ induces a subtournament isomorphic to $(L,<)$ and such that all points from $V_1$ are  on the quarter-circle of $S(2)$ in the first quadrant. Let $V_2:=V(S(2))\setminus V_1$. Then the age of the subtournament of $S(2)$ induced by $V_2$ is equal to $\Age(S(2))$ (to see this, the construction $w\mapsto T_w$ from above may be used, observing that the $p_i$ may always be chosen from the fourth quadrant of $S(2)$, so that neither of them is in $V_1$). Now define a function $\chi_L\colon V(S(2))\to \{1,2\}$ through 
	\[
	\chi_L\colon v\mapsto
	\begin{cases}
		1 &\text{if } v\in V_1,\\
		2 &\text{else.}
	\end{cases}
	\] 
	Then $\Age(S(2)[\chi_L])=\Age(S(2)[I_2])$. Moreover, if  $A$ is the subtournament of $S(2)[\chi_L]$ that is induced by all those vertices whose  equivalence class  with respect to the non-arc relation  is a singleton, then $A$ is isomorphic to $(L,<)$. Thus, if $(L_1,<)$ and $(L_2,<)$ are  non-isomorphic countable chains, then $S(2)[\chi_{L_1}]\not\cong S(2)[\chi_{L_2}]$. By the Cantor-Bernstein-Theorem (see \cite[\S 6, Theorem I]{Ber05}) there are continuum many isomorphism classes of countable linear orders.
	
	\subsection*{About class~\ref{class9}}
	First let us show that there are continuum many distinct ages of countable HH oriented graphs from class~\ref{class9}. For this we are going to use the infinite antichain $(\fB_i)_{1\le i<\omega}$ of finite tournaments that was discovered by Henson in \cite[pages 497, 498]{Hen72}. For convenience, let us recall the construction: Let $n$ be a positive integer. Define a tournament $\fA_n$ on the vertex set $\{0,1,\dots,n+2\}$ and with arc relation given by:
	\[
		a\to b :\iff (b=a+1)\lor (b+1<a). 
	\] 
	Now $\fB_n$ is obtained from $\fA_n$ by adjoining an additional vertex $\alpha$ and by defining:
	\[
		\alpha\to 0,\,\alpha\to n+2,\,\text{and } \forall \, 0< b<n+2\,:\, b\to\alpha.
	\]
	 What is important for us in the sequel is that for every $n$ we have that $\fB_n$ is strongly connected (in other words, it contains an oriented Hamiltonian cycle).
	
	Let $\tn=(n_j)_{j<\omega}$ be a monotonous sequence of positive integers. Consider the subsequence $(\fB_{n_j})_{j<\omega}$ of $(\fB_i)_{1\le i<\omega}$.  Define a countable tournament $\fT_{\tn}$ with vertex set $\{(j,v)\mid j<\omega,\, v\in V(\fB_{n_j})\}$,  where the arc set is given by
	\[
		(j_1,v_1)\to(j_2,v_2):\iff (j_1< j_2)\lor (j_1=j_2\land v_1\to v_2).
	\]
	Note now that, since $(\fB_i)_{1\le i<\omega}$ is an antichain consisting of  strongly connected tournaments, it holds that
	\[
	\forall i\in\omega\setminus\{0\}\,:\, \bigl[\fB_i\in\Age(\fT_{\tn}) \iff \exists j<\omega\,:\, \fB_i=\fB_{n_j}\bigr].
	\]
	In particular, every change in the sequence $\tn$  leads to a change of the age of $\fT_{\tn}$.
	
	Next we use that $T^\infty$ is universal for countable tournaments and embed $\fT_{\tn}$ into $T^\infty$. For simplicity, let us assume that $\fT_{\tn}$ is actually a subtournament of $T^\infty$. 
	 Consider now the function $\chi_{\tn}\colon V(T^\infty)\to(\omega\cup\{\omega\})\setminus\{0\}$ given by:
	\[
	\chi_{\tn}\colon v\mapsto
	\begin{cases}
		2 & \text{if } v\in V(\fT_{\tn}),\\
		1 & \text{else.}	
	\end{cases}
	\]  
	It remains to observe that 
	\[
	\forall i\in\omega\setminus\{0\}\,:\,\bigl[ \fB_i\in\Age(\fT_{\tn}) \iff \fB_i[I_2]\in\Age(T^\infty[\chi_{\tn}])\bigr].
	\]
	In particular, different choices for $\tn$ lead to different ages of  $T^\infty[\chi_{\tn}]$. Thus, we showed that there are continuum many pairwise distinct ages of HH oriented graphs from class~\ref{class9} (in fact all of them are contained in $\Age(T^\infty[I_2])$).
	
	Next, let us show that there are continuum many pairwise non-isomorphic countable HH oriented graphs whose age is equal to  $\Age(T^\infty[I_\omega])$.
	First note that $V(T^\infty)$ may be partitioned into two subsets $V_1$ and $V_2$, such that $V_1$ induces a subtournament  isomorphic to $T^\infty$ and such that $V_2$ is infinite. To see this, let $\widehat{T}^\infty$ be the tournament obtained from $T^\infty$ by adjoining to it a  vertex $u$ that has arcs to all other vertices of $T^\infty$. Since $T^\infty$ is universal, it has a  subtournament that is isomorphic to $\widehat{T}^\infty$. Without loss of generality, we may assume that  $\widehat{T}^\infty$ is actually a subtournament of $T^\infty$. Then its complement must be infinite (recall that removing finitely many vertices from $T^\infty$ leaves us with a tournament isomorphic to $T^\infty$). Let $V_1:=V(\widehat{T}^\infty)\setminus\{u\}$, and let $V_2:=V(T^\infty)\setminus V_1$. Let $M$ be an infinite subset of $\omega\setminus\{0\}$ and let $f\colon V_2\to M$ be a bijection. Consider the function $\chi_M\colon V(T^\infty)\to \omega\setminus\{0\}$ given by:
	\[
		\chi_M\colon v\mapsto\begin{cases}
			f(v) & \text{if } v\in V_2\\
			\omega & \text{else.}
		\end{cases}
	\]
	Then the set of sizes of equivalence classes of the non-arc relation $\not\sim$ of $T^\infty[\chi_M]$ is equal to $M\dotcup\{\omega\}$. 
	The union of all the infinite equivalence classes induces an oriented subgraph isomorphic to $T^\infty[I_\omega]$ in $T^\infty[\chi_M]$. In particular, $\Age(T^\infty[\chi_M])=\Age(T^\infty[I_\omega])$.  Moreover, it is clear that whenever $M_1\neq M_2$, then $T^\infty[\chi_{M_1}]\not\cong T^\infty[\chi_{M_2}]$. 

\section{Concluding remarks}\label{future}
We have succeeded in classifying the countable homomorphism homogeneous and polymorphism homogeneous oriented graphs with an irreflexive arc-relation. A natural question is whether this classification can be extended to  oriented graphs in which loops are allowed. 

Here, cores do not play a role as any oriented graph with at least one loopy vertex has the trivial core consisting of a single vertex with a loop. However, this does not imply that a classification attempt is out of reach. We have already shown that any countable homomorphism homogeneous oriented graph with at least one loop has either a transitive arc-relation or is isomorphic to $C_3$ with all loops added. 

A complete list of possible minimal witnesses for oriented graphs with transitive, antisymmetric arc-relation was given in \cite{MPPhD}. So instead of using cores, the classification may be approached by stratifying along this list of minimal witnesses. While ambitious, such a task appears to be just on the edge of feasibility.

\section*{Acknowledgements}
We are grateful to the anonymous referees for their  constructive remarks, which greatly  helped to improve  the presentation of the paper and saved us from a serious oversight.

\end{document}